\documentclass[11pt,a4paper]{article}

\usepackage{amssymb}
\usepackage{amsmath, amsthm}
\usepackage{arydshln}
\usepackage{epsfig}
\usepackage{setspace}
\usepackage{comment} 
\usepackage[margin=1in]{geometry}

\def\RR{{\mathbb{R}}}

\def\CC{{\mathbb{C}}}

\numberwithin{equation}{section}
\newcommand{\spec}{\mathop{\rm spec} }

\newcommand{\dom}{\mathop{\rm dom} }

\newcommand{\Conv}{\mathop{\rm Conv} }

\newcommand{\argmin}{\mathop{\rm argmin} }

\newcommand{\diag}{\mathop{\rm diag} }
\newcommand{\trace}{\mathop{\rm tr} }

\newtheorem{Thm}{Theorem}[section]
\newtheorem{Prop}[Thm]{Proposition}
\newtheorem{Lem}[Thm]{Lemma}
\newtheorem{Cor}[Thm]{Corollary}
\newtheorem{Question}[Thm]{Question}
\theoremstyle{definition}
\newtheorem{Prob}[Thm]{Problem}
\newtheorem*{Clm}{Claim}

\newtheorem{Rem}[Thm]{Remark}

\begin{document}

\title{Gradient descent for unbounded convex functions \\ on Hadamard manifolds \\ and its applications to scaling problems\footnote{The conference version appeared as 
H. Hirai and K. Sakabe: Gradient descent for unbounded convex functions on Hadamard manifolds and its applications to scaling problems. 
{\em 65th IEEE Symposium on Foundations of Computer Science (FOCS 2024)}, pp.2387--2402}
}

\author{Hiroshi Hirai\thanks{Graduate School of Mathematics, Nagoya University, Nagoya, 464-8602, Japan. hirai.hiroshi@math.nagoya-u.ac.jp} \and Keiya Sakabe \thanks{Faculty of Physics, Ludwig-Maximilians-Universit\"{a}t M\"{u}nchen, Munich, Germany. keiya.sakabe@lmu.de}}

\maketitle

\begin{abstract}
    In this paper, we study the asymptotic behavior of continuous- and discrete-time gradient flows of 
    a ``lower-unbounded" convex function $f$ on a Hadamard manifold $M$, particularly,  
    their convergence properties to the boundary $M^{\infty}$ at infinity of $M$.
    We establish a duality theorem that the infimum of the gradient-norm $\|\nabla f(x)\|$ of $f$ 
    over $M$ 
    is equal to the supremum of the negative of the recession function $f^{\infty}$ 
    of $f$ over the boundary $M^{\infty}$, provided the infimum is positive.
    Further, 
    the infimum and the supremum are
    obtained by the limit of the gradient flow of $f$. 
     Our results feature convex-optimization ingredients of the moment-weight inequality 
     for reductive group actions by Georgoulas, Robbin, and Salamon, 
     and are applied to noncommutative optimization by B\"urgisser et al. FOCS 2019.
    We show that gradient descent of 
    the Kempf-Ness function for an unstable orbit
    converges to a destabilizing 1-parameter subgroup in the Hilbert-Mumford criterion, and 
    the associated moment-map sequence converges to the minimum-norm point of the moment polytope.
    We show further refinements for operator scaling---the left-right action 
    on a matrix tuple $A= (A_1,A_2,\ldots,A_N)$.
    We characterize the gradient-flow limit of operator scaling
    by a vector-space generalization of 
    the classical Dulmage-Mendelsohn decomposition of a bipartite graph.
    For a special case of $N = 2$, 
    we reveal that the limit determines 
    the Kronecker canonical form of a matrix pencil $s A_1+A_2$.
\end{abstract}


MSC-classifications: 90C25, 53C35

OR/MS subject classifications: Programming/Nonlinear/Convex, Mathematics/Convexity

\section{Introduction}

In convex optimization, it is typically assumed that the objective function $f$ is bounded below. 
The performance of a minimization algorithm is evaluated 
by its convergence behavior to the minimum of $f$. 
This paper addresses the convergence behavior of minimization algorithms 
for a “lower-unbounded” convex function $f$, i.e., $\inf f(x) = - \infty$. 
This may look meaningless, because
the trajectory $x_i$ of an algorithm diverges to infinity, 
and $f(x_i)$ goes to $-\infty$.
The meta question of the paper is: 
\begin{center}
{\it What can we gain from such a divergent sequence?}
\end{center}

Let us formalize our setting and mention its background. 
Let $M$ be a Hadamard manifold---a simply-connected complete Riemannian manifold 
with nonpositive sectional curvature.
Let $f:M \to \RR$ be a (twice differentiable) geodesically convex function, that is, 
$f$ is convex along any geodesic.  
We consider the following unconstrained convex optimization problem on $M$:
\begin{equation}\label{eqn:problem}
{\rm inf.}\quad f(x) \quad \mbox{s.t.}\quad  x \in M, \quad \mbox{where $f$ can be lower-unbounded.} 
\end{equation}
Such a problem setting is significant in the recent progress 
on {\em operator scaling}~\cite{Gurvits2004} and generalizations; see \cite{AGLOW_STOC2018,BFGOWW_FOCS2018,BFGOWW_FOCS2019,GGOW_BrascampLieb,GGOW,GargOliveira2018,HiraiNieuwboerWalter2023FOCS}. 
In the classical {\em matrix scaling}~\cite{Sinkhorn1964}, 
the scalability is equivalent to the boundedness of (\ref{eqn:problem}) for some convex function $f$ in $\RR^n$. 
Further, it is also equivalent to the perfect-matching condition of the associated bipartite graph. 
Hayashi, Hirai, and Sakabe~\cite{HayashiHiraiSakabe} 
studied asymptotic behavior of the Sinkhorn algorithm for the unscalable (unbounded) case, 
and revealed  
that a combinatorial certificate ({\em Hall blocker}) of unscalability 
can be identified from divergent behavior of the Sinkhorn algorithm. 
Although a Hall blocker is easily obtained by network-flow algorithms, 
finding the corresponding certificate ({\em shrunk subspace}) 
for the operator scaling setting is possible but quite difficult; 
see \cite{HamadaHirai2021,IQS2017,IQS2018}.
Just recently,
Franks, Soma, and Goemans~\cite{FranksSomaGoemans_SODA2023} 
modified the {\em operator Sinkhorn algorithm}---an alternating minimization algorithm 
for some convex function on the Hadamard manifold 
of positive definite matrices---to obtain a shrunk subspace in polynomial time, 
although it is still rather complicated.
The matrix and operator scaling problems are generalized to
a class of convex optimization involving reductive group actions, 
called {\em noncommutative optimization}~\cite{BFGOWW_FOCS2019}, which asks
to minimize the {\em Kempf-Ness function} associated with an orbit of the action.
This is formulated as a convex optimization problem on a representative class of Hadamard manifolds---{\em symmetric spaces of nonpositive curvature}.
It is lower-unbounded if and only if the orbit is unstable, where 
a $1$-parameter subgroup ({\em destabilizing 1-PSG}) in the {\em Hilbert-Mumford criterion} 
is the unboundedness certificate that generalizes a Hall blocker and a shrunk subspace.
As mentioned in \cite{BFGOWW_FOCS2019}, it is a great challenge 
to design polynomial-time algorithms for several noncommutative optimization problems, 
such as (un)stability determination, moment-polytope membership, and orbit-closure intersection, which will bring fruitful applications to 
broader areas of mathematical sciences. 
Many of them involve (un)bounded determination of Kempf-Ness functions, 
though our current knowledge on such problems is limited.

Motivated by these considerations, we study minimization of lower-unbounded 
convex functions on Hadamard manifolds.
Even in the Euclidean setting $M = \RR^n$, 
there are few works (see e.g., \cite{Auslender1997,Obuchowska2004}) on such study.
We focus on asymptotic behavior of
the simplest algorithm---{\em gradient descent}.
Accompanied with this, 
we also consider its continuous version---{\em gradient flow}, that is,  
a trajectory produced by the differential equation $\dot x(t) = - \nabla f(x(t))$. 

The contributions and organization of this paper are summarized as follows.
We begin with a general study of the asymptotic behavior of the gradient flow/descent 
for an unbounded convex function $f$ on a Hadamard manifold $M$.
As in the Euclidean setting, the {\em recession function} ({\em asymptotic slope}) 
$f^{\infty}$ of $f$ (see \cite{Hirai_Hadamard2022,KLM2009JDG}) is a basic tool of analyzing unboundedness, which 
is a function defined on the {\em boundary $M^{\infty}$ at infinity of $M$}.
Intuitively, the boundary $M^{\infty}$ is the set of all {\em directions} $\xi$ 
from an arbitrary fixed point $x_0$, and $f^{\infty}(\xi)$ represents the slope of $f$ along the direction $\xi$ at infinity.
Then, Hadamard manifold $M$ admits compactification $M \cup M^{\infty}$, where 
the resulting topology is called the {\em cone topology}.
These notions and related manifold terminologies are summarized in Section~\ref{sec:preliminaries}.

We focus on convergence properties, with respect to the cone topology, 
of the gradient flow/descent for an unbounded convex function $f$. 
In Section~\ref{sec:gradientflow}, 
under a sufficient condition $\inf_{x \in M } \|\nabla f(x)\| > 0$ of unboundedness,
we establish in Theorem~\ref{thm:continuous_main} that
the gradient flow $x(t)$ converges to a point of boundary $M^{\infty}$ 
with providing the following min-max (inf-sup) relation: 
\begin{equation}\label{eqn:duality_intro}
\lim_{t \to \infty} \|\nabla f(x(t)) \| = \inf_{x \in M} \|\nabla f(x)\| 
= \sup_{\xi \in M^{\infty}} -f^{\infty}(\xi) = - f^{\infty}\left(\lim_{t \to \infty} x(t)\right). 
\end{equation}
The limit $\lim_{t \to \infty} x(t)$ is the unique minimizer of $f^{\infty}$ over $M^{\infty}$, and 
is a certificate of unboundedness.
Further, we also show in Theorem~\ref{thm:discrete_main} that the same result holds 
for the sequence $x_i$ produced by gradient descent applied to an $L$-smooth convex function $f$ 
with step-size $1/L$.
These are the core results of the paper that drive the subsequent arguments. 

Even in the Euclidean setting $M=\RR^n$, 
these convergence results on the gradient flow/descent seem new, and  
bring an interesting ramification (Theorem~\ref{thm:mnp}): 
both $\nabla f(x(t))$ and $\nabla f(x_i)$
converge to the minimum-norm point $p^*$ of 
the gradient space $\overline{\nabla f(\RR^n)}$ (that is convex).
This means that gradient descent is interpreted as 
a minimum-norm point algorithm in the gradient space. 
Other interesting connections and implications to {\em Hessian Riemannian gradient flow}~\cite{AlvarezBolteBrahic2004}, 
{\em mirror descent}~\cite{NemirovskyYudin},  
and geometric programming are also mentioned.

In Section~\ref{sec:applications}, we present applications.
In Section~\ref{subsec:norm_minimization}, we deal 
with the norm-minimization problem for a reductive group action on a complex vector/projective space.
As mentioned, this is the problem of minimizing the Kempf-Ness function $f_v$ associated with an orbit of the action. 
Then, gradient descent is essentially the {\em first-order algorithm} in \cite{BFGOWW_FOCS2019}. 
Applying our results, we show 
that for the unstable case
the trajectory of the first-order algorithm
converges, in cone topology, to the unique minimizer of $f_v^{\infty}$, 
that yields a destabilizing $1$-PSG in the Hilbert-Mumford criterion. 
Further, the spectrum of the moment-map ($=$ transported gradient of $f_v$) 
along the trajectory 
converges to the minimum-norm point of the moment polytope~$\varDelta_v$.
For the gradient-flow setting,  
we reveal the connection to 
the theory of the {\em moment-weight inequality} 
for reductive group actions,  
developed by Georgoulas, Robbin, and Salamon~\cite{GRS_Book} 
building upon the earlier work by Kempf, Kirwan, Mumford, and Ness in GIT
and the recent work by Chen and Sun~\cite[Section 4]{ChenSun2014} in $K$-stability.
Specifically, the weak duality $\|\nabla f(x)\|\geq - f^{\infty}(\xi)$ in (\ref{eqn:duality_intro}) 
becomes the moment-weight inequality, and the strong duality via the gradient flow 
can explain important parts of their theory.
It may be fair to say that our results in Section~\ref{sec:gradientflow} 
extract and discretize convex-optimization ingredients of their theory.

In Section~\ref{subsec:operatorscaling}, we focus on the {\em left-right action} 
$SL_n(\CC) \times SL_m(\CC) \ni (g,h) \mapsto gAh^{\dagger}$
on a matrix tuple $A= (A_1,A_2,\ldots,A_N)$, 
that corresponds to the operator scaling problem.
In this setting, the middle equality in (\ref{eqn:duality_intro}) 
is interpreted as a duality theorem for the scalability limitation (Theorem~\ref{thm:scalability_limit}),
which sharpens Gurvits' characterization in the inf-sup form.
We then study the limit of the gradient flow/descent for 
the Kempf-Ness function $(g,h) \mapsto \log \|gAh^{\dagger}\|$. 
Our focus is in the unscalable case, 
whereas the scalable case was studied in detail by Kwok, Lau, and Ramachandran~\cite{KwokLauRamchandran_SICOMP2021}.    
We show in Theorems~\ref{thm:mnp_Q_A} and \ref{thm:limit_in_P(V)} 
that the minimum-norm point of the moment polytope $\varDelta_A$ and 
the limit of the gradient flow/descent are characterized by a certain simultaneous block-triangularization 
of $A= (A_1,A_2,\ldots,A_N)$, which is a vector-space generalization of 
    the classical {\em Dulmage-Mendelsohn decomposition}~\cite{DulmageMendelsohn1958} 
    of a bipartite graph.
    More specifically, the sequence of (normalized) scaling tuples $g_kAh_k^{\dagger}/\|g_kAh_k^{\dagger}\|$ 
    along the gradient descent converges to
    a block-diagonal matrix modulo the left-right unitary group action, 
    where the block structure is determined by our generalized DM-decomposition.
This answers the gradient-descent variant of an open question 
by Garg and Oliveira~\cite[Section 6]{GargOliveira2018} 
for asking asymptotic behavior of the operator Sinkhorn algorithm for unscalable instances.
Finding this block structure itself is significant. 
We partially eliminate the unitary indeterminacy from $g_kAh_k^{\dagger}$, 
and exploit a convergent sequence to a coarse block-triangular structure (Theorem~\ref{thm:convergence_to_cDM}). 
This leads to a new construction of a shrunk subspace (certificate of unscalability) 
by gradient descent combined with the rounding procedure in Franks, Soma, and Goemans~\cite{FranksSomaGoemans_SODA2023}.    

    In Section~\ref{subsec:pencil}, for a special case of $N = 2$, 
    we reveal that our DM-decomposition of $(A_1,A_2)$ coarsens and determines 
    the well-known {\em Kronecker canonical form} of a matrix pencil $s A_1+A_2$.
    The Kronecker form plays important roles in systems analysis 
    by a differential-algebraic equation (DAE) $A_1 \dot u(t) + A_2u(t) = 0$.
    Its computation has been studied for a long time 
    in the literature of numerical computation; see e.g., \cite{DemmelKagstrom1993,VanDooren1979}. 
Our convergence result (Theorem~\ref{thm:convergence_to_cKronecker}) 
suggests a new iterative method for determining the Kronecker structure, 
which is based on simple gradient descent and is conceptually different from the existing ones.

These results may be positioned as attempts of detecting, by algorithms in $M$, 
hidden structures in the boundary $M^{\infty}$ at infinity,
which has been little studied so far. 
We hope that our attempts lead to more serious studies from computational complexity perspective. 
Particularly, it is an important future direction to 
improve the present convergence-type results to 
the ones having explicit iteration complexity.

After the submission of this paper, 
there have been several subsequent developments~\cite{Hirai2025entanglement,Hirai2025,Ohta2025,Sakabe,Sakabe2026supportfunctional}.

\section{Preliminaries}\label{sec:preliminaries}
Let $\RR$ and $\RR_+$ denote the sets of real and nonnegative real numbers, respectively.
We often add to $\RR$ and $\RR_+$ 
the infinity elements $\pm \infty$, 
where the topology and ordering $\leq$ 
are extended in the usual way.
Let $\CC$ denote the set of complex numbers $z = x+iy$, 
where $\bar z$ denotes the complex conjugate $x-iy$ and $Re z$ denotes the real part $x$.
The same notation is used for a complex vector $\zeta = u+iv \in \CC^n$ with $u,v \in \RR^n$ 
as $\bar \zeta = u - iv$.
For a matrix $A$ over $\CC$, let $A^{\dagger}$ 
denote the transpose conjugate. 
For sets $I$ and $J$ of row indices and column indices of $A$, 
let $A[I,J]$ denote the submatrix of $A$ with row indices in $I$ and column indices in $J$.
For two matrices $A,B$ (of possibly different sizes), 
let $A \oplus B$ denote the block-diagonal matrix 
of block-diagonals $A,B$ in order.
For a vector $p \in \RR^n$, let $\diag p$ denote
the $n \times n$ diagonal matrix with $(\diag p)_{ii} = p_i$.

The general linear group $GL(n,\CC)$ and the special linear group $SL(n,\CC)$ over $\CC$ are simply denoted by $GL_n$ and $SL_n$, respectively. 
The unitary group $U(n)$ and the special unitary group $SU(n)$ are denoted by $U_n$ and $SU_n$, respectively.
For a finite-dimensional vector space $V$ over $\CC$, 
let $GL(V)$ denote the group of linear isomorphisms on $V$.

For a positive integer $n$, let $[n] := \{1,2,\ldots,n\}$.
For $X \subseteq [n]$, let ${\bf 1}_X \in \RR^n$ be defined by $({\bf 1}_{X})_i = 1$ if $i \in X$ and $0$ otherwise, where ${\bf 1}_{[n]}$ is simply written as ${\bf 1}$.

A sequence $(x_i)_{i=0,1,\ldots,}$ and function $(x(t))_{t \in \RR_+}$ are simply 
denoted by $x_i$ and $x(t)$, respectively. 
For a real-valued sequence $a_i$ and continuous function $h(t)$, 
we will use several times the following:
\begin{eqnarray}
&&  \liminf_{i \to \infty} a_i \leq \liminf_{i \to \infty} \frac{1}{i} \sum_{j=1}^{i} a_j \leq   \limsup_{i \to \infty} \frac{1}{i}\sum_{j=1}^i a_j \leq \limsup_{i \to \infty} a_i, \label{eqn:exercise1}\\
&& \liminf_{t \to \infty} h(t) \leq \liminf_{t \to \infty} \frac{1}{t} \int_{0}^t h(s) ds \leq   \limsup_{t \to \infty} \frac{1}{t} \int_{0}^t h(s) ds \leq \limsup_{t \to \infty} h(t). \label{eqn:exercise2}
\end{eqnarray}
This is a little exercise in calculus. For example, 
the leftmost $\leq$ in (\ref{eqn:exercise2}) follows from: 
Suppose that $\alpha := \liminf_{t \to \infty} h(t) \in \RR$. 
Then $\forall \epsilon > 0$, $\exists N \geq 0$, $\forall t \geq N$, $h(t) \geq \alpha - \epsilon$, and hence $\forall t \geq N$, $\frac{1}{t}\int_0^t h(s)ds 
\geq \frac{1}{t}\int_0^{N} h(s)ds + \frac{t-N}{t}(\alpha - \epsilon) 
\underset{t \to \infty}{\longrightarrow} \alpha - \epsilon$. 
Since $\epsilon$ is arbitrary, we have $\liminf_{t \to \infty} \frac{1}{t} \int_{0}^t h(s) ds \geq \alpha$.
\subsection{Riemannian geometry}
We will utilize standard terminologies and notation on 
Riemannian geometry; see e.g., \cite{Sakai1996}. See also a recent book~\cite{Boumal_Book} 
for optimization perspective. 
We assume sufficient differentiability 
for manifolds, functions, maps, and vector/tensor fields on them.
Let $M$ be a Riemannian manifold. 
For $x \in M$, let $T_x = T_x(M)$ 
denote the tangent space of $M$ at $x$, where  
$\langle \cdot , \cdot \rangle = \langle \cdot , \cdot \rangle_x$ 
denotes the Riemannian metric at $x$ and $\| \cdot \| := \sqrt{\langle \cdot , \cdot \rangle}$ 
denotes the associated norm.
Let $S_x := \{u \in T_x \mid \|u\| = 1\}$ and  $B_x := \{u \in T_x \mid \|u\| \leq 1\}$ denote 
the unit sphere and ball in $T_x$, respectively.
The angle $\angle (u,v)$ of two vectors $u,v \in T_x$ is defined 
as $\cos^{-1}(\langle u,v \rangle/\|u\|\|v\|)$.
The product $M \times M'$ of two Riemannian manifolds $M,M'$ is viewed as a Riemannian manifold 
by setting $\langle (u,u'),(v,v') \rangle_{(x,x')} := \langle u,v \rangle_x + \langle u',v' \rangle_{x'}$.

For a path $\gamma:[a,b] \to M$ and $t \in [a,b]$,  
let $\dot \gamma(t)$ 
denote the tangent vector of $\gamma$ at $T_{\gamma(t)}$.
The length of the path $\gamma$ 
is defined by $\int_a^b \|\dot \gamma(t)\| dt$.
The distance $d(x,y)$ between $x,y  \in M$ 
is the infimum of the length of a path connecting $x$ and $y$. 
We consider the Levi-Civita connection $\nabla$ associated with the Riemannian metric.
The connection $\nabla$ determines the {\em parallel transport}  
$\tau_{\gamma}^t:T_{\gamma(0)} \to T_{\gamma(t)}$ 
along any path $\gamma:[0,b] \to M$ with $t \in [0,b]$, 
where $\tau_\gamma^{-t} := (\tau_\gamma^t)^{-1}$.
By using the parallel transport, 
the covariant derivative $\nabla_u V$ of
a vector field $V = (V_x)_{x \in M}$ by $u \in T_x$ is given by
$\nabla_u V := 
(d/dt) \tau^{-t}_{\gamma} V_{\gamma(t)} \mid_{t = 0}$, 
where $\gamma$ is a path with $\gamma(0) =x$ and $\dot \gamma(0) = u$.

In this paper, any manifold $M$ is assumed to be complete. 
That is, the metric space $(M,d)$ is complete. 
Then, the distance $d(x,y)$ is always attained 
by a {\em geodesic}---a path $\gamma:[a,b] \to M$ 
satisfying $\nabla_{\dot \gamma(t)}\dot \gamma = 0$ for $t \in [a,b]$.
By a {\em unit-speed geodesic ray} we mean a geodesic $\gamma:[0,\infty) \to M$ 
with $\|\dot \gamma(0)\| = 1$ (and then $\|\dot \gamma(t)\| = 1$ for all $t$).
For $x \in M$ and $u \in T_x$, 
there is a unique geodesic 
$\gamma (t)$ with $\gamma(0) = x$ and $\dot \gamma(0) = u$, denoted by $\exp_{x} tu$.
By completeness of $M$, the map $t \mapsto \exp_{x} tu$ is defined on $\RR_+$.
This gives rise to  a surjective map $\exp_x: T_x \to M$, called the {\em exponential map}.

For a map $\varphi: M \to N$, where $N$ is another manifold, 
let $d\varphi: T_x(M) \to T_{\varphi(x)}(N)$ denote the differential of $\varphi$ 
at $x \in M$. 
The differential $df = df_x : T_x \to \RR$ of a function $f: M \to \RR$ is given by
$df(u) = (d/dt) f(\gamma(t)) \mid_{t=0}$,
where $\gamma$ is a path with $\gamma(0) = x$ and $\dot \gamma(0) = u \in T_x$.
The gradient $\nabla f(x) \in T_x$ of $f$ is then defined via
\[
\langle \nabla f(x),u \rangle := df(u) \quad (u \in T_x). 
\]
The covariant differentiation of the gradient vector field $\nabla f$ 
gives rise to the {\em Hessian} $\nabla^2 f(x): T_x \to T_x$:
\begin{equation}\label{eqn:Hessian}
   \nabla^2 f(x) u := \nabla_u \nabla f(x) \quad (u \in T_x). 	
\end{equation}
The Hessian is a symmetric operator 
in the sense that $\langle \nabla^2 f(x) u,v \rangle = \langle \nabla^2 f(x) v,u \rangle$.

\paragraph{Complex projective space.} 
We will consider the complex projective space as a Riemannian manifold.
Let $V$ be an $n$-dimensional vector space over $\CC$.
The {\em complex projective space} $\mathbb{P}(V)$ over $V$ 
is a quotient manifold $V \setminus \{0\}/\sim$ by the equivalent relation 
$v \sim v' \Leftrightarrow v = \alpha v'$ $(\exists \alpha \in \CC \setminus \{0\})$.  
The image of $v$ by $V\setminus \{0\} \to \mathbb{P}(V)$ is denoted by $[v]$.
A Riemannian structure on $\mathbb{P}(V)$ is given by 
the {\em Fubini-Study form} as follows. 
Let $(\cdot,\cdot)$ be a Hermitian inner product on $V$.
Regard $V$ as a $2n$-dimensional Euclidean space by the real inner product $Re (\cdot,\cdot)$. 
This induces a Riemannian structure on the sphere $S^{2n-1} = \{v \in V \mid \|v\| = r\}$, where we set $r := \sqrt{2}$.
Further, $U_1 (= U(1))$ acts isometrically on $S^{2n-1}$ 
by scalar multiplication $U_1 \times S^{2n-1} \ni (e^{i\theta},v) \mapsto e^{i \theta}v$.
Then $\mathbb{P}(V)$ is viewed as the Riemannian quotient of $S^{2n-1}$ with respect to this action. 
The resulting metric on $\mathbb{P}(V)$ is called the {\em Fubini-Study} metric.
See e.g., \cite[Chapter 9]{Boumal_Book} for Riemannian quotient manifolds.

\subsection{Hadamard manifold}\label{subsec:Hadamard_manifold}

A {\em Hadamard manifold} $M$ is a simply-connected complete 
Riemannian manifold having nonpositive sectional curvature 
everywhere; see \cite[V.4]{Sakai1996}. For any two points in $M$, a geodesic connecting them is uniquely determined (up to affine rescaling). 
The exponential map $\exp_x$ 
is a diffeomorphism from $T_x$ to $M$.
The parallel transport from $T_x$ to $T_y$
along the geodesic is simply denoted by $\tau_{x \to y}$.

In this paper, the {\em boundary $M^{\infty}$ at infinity} and the {\em cone topology} on $M \cup M^{\infty}$ play particularly important roles;
see \cite[V.4.2]{Sakai1996} for a quick introduction to these notions. 
Two unit-speed geodesic rays $\gamma,\gamma':\RR_+ \to M$
are called {\em asymptotic} if $d(\gamma(t),\gamma'(t)) < C$ $(t \in \RR_+)$ for some constant $C>0$. The asymptotic relation is 
an equivalence relation on the set of all unit-speed geodesic rays.
Let $M^{\infty}$ denote the set of all equivalence classes. 
%
Let us fix an arbitrary point $x \in M$.
Any unit vector $u \in S_{x}$ defines  
an asymptotic class of unit-speed geodesic ray $t \mapsto \exp_{x} t u$.
This correspondence is a bijection between $S_{x}$ and $M^{\infty}$, and  
induces a topology on $M^{\infty}$ that is isomorphic to the sphere $S_{x}$. In fact, this topology is independent of the choice of $x$.
Further, the topologies on $M$ and on $M^{\infty}$ are 
extended to $M \cup M^{\infty}$ as follows.
Since $\exp_x$ is a diffeomorphism, 
it holds $M \simeq (S_x \times \RR_+)/\sim_0$, where $\sim_0$ is 
the equivalence relation defined by $(u,r)\sim_0 (u',r')$ $\Leftrightarrow$ 
$(u,r)= (u',r')$ or $r=r' = 0$.
With $M^{\infty} \simeq S_x \times \{\infty\}$,  
we obtain the compact Hausdorff space $M \cup M^{\infty} \simeq (S_x \times (\RR_+ \cup \{\infty\}))/\sim_0$ (isomorphic to $B_x$).
This topology on $M \cup M^{\infty}$ is called the {\em cone topology}.
In this topology, a sequence $x_i$ in $M$ converges to $\xi \in M^{\infty}$ if and only if
\begin{itemize}
\item $d(x,x_i) \to \infty$, and  
\item  the sequence $u_i$ in $S_x$  
determined by $x_i= \exp_x d(x,x_i)u_i$ converges to $u \in S_x$, where the asymptotic class of geodesic $t \mapsto \exp_x tu$ is equal to $\xi$. 
\end{itemize}

The {\em angle} $\angle^{\infty} (\xi,\xi')$ of two points $\xi,\xi' \in M^{\infty}$ is defined 
as $\sup_{x \in M} \angle(u,u')$, where $u$ and $u'$ are the 
representatives of $\xi$ and $\xi'$, respectively, at $T_x$.
The angle defines a metric on $M^{\infty}$, which induces a different topology.
By using the angle metric on $M^{\infty}$, we can define a metric $d^{\infty}$ on 
the {\em Euclidean cone} $CM^{\infty} := (M^{\infty} \times \RR_+)/\sim_0$ of the boundary $M^{\infty}$
by $d^{\infty}((\xi,r),(\xi',r'))^2 = r^2+(r')^2 - 2rr' \cos \angle^{\infty} (\xi,\xi')$. 
This space $CM^{\infty}$ is viewed as the space of asymptotic classes of (not necessarily unit-speed) 
geodesic rays. It is identified with $T_x$, 
though the metric space $(CM^{\infty}, d^{\infty})$ has 
a different topology from $T_x$ and is not necessarily a manifold.
This metric space $(CM^{\infty}, d^{\infty})$ is a {\em Hadamard space}---a complete geodesic metric space satisfying the {\em CAT(0)-inequality}~\cite{BridsonHaefliger1999}.
It is uniquely geodesic, and its convexity is defined along geodesics.
The unit ball $B^{\infty} = \{p \in CM^{\infty} \mid d^{\infty}(0,p) \leq 1\}$ around the origin $0$ is a convex set, 
where the origin $0$ is the image of point $(\xi,0)$. 
Observe that $B^{\infty}$ can be identified with $B_x$ for any $x \in M$.

\paragraph{Manifold of positive definite matrices and symmetric space.}
A representative example of a Hadamard manifold is
the space $P_n$ of $n\times n$ positive definite Hermitian matrices; see \cite[II.10]{BridsonHaefliger1999}. 
The tangent space $T_x$ at $x \in P_n$ is identified 
with the real vector space $\mathfrak{p}_n$ of Hermitian matrices, 
and the Riemannian metric   
is given by $\langle G,H\rangle_x := \trace x^{-1}Hx^{-1}G$.
In this space, 
several manifold notions are explicitly written; 
see e.g., \cite[Section 5.2]{HiraiNieuwboerWalter2023FOCS}.
The exponential map $\exp_x$ at $x$ is given 
by $H \mapsto x^{1/2}e^{x^{-1/2}Hx^{-1/2}}x^{1/2}$, where 
$e^{\bullet}$ is the matrix exponential.
Particularly, any geodesic issuing at $x$ is of form $t \mapsto x^{1/2}e^{tx^{-1/2}Hx^{-1/2}}x^{1/2}$ for some Hermitian matrix $H \in T_x$ with $\|H\| = \|x^{-1/2}Hx^{-1/2}\|_{\rm F} = 1$, where $\|\cdot \|_{\rm F}$ is the Frobenius norm.
An explicit formula of the geodesic parallel transport 
$\tau_{x \to y}$ is also known.
We will use one special case: $\tau_{x \to I}H = x^{-1/2}H x^{-1/2}$.

Any totally geodesic subspace $M$ of $P_n$ is also a Hadamard manifold.
Here, a submanifold $M \subseteq P_n$ is said to be 
{\em totally geodesic} if every geodesic in $M$ is also geodesic in $P_n$. 
It is known \cite[II.10.58]{BridsonHaefliger1999} that for a connected Lie group $G \subseteq GL_n$ defined by polynomials and satisfying $G =G^{\dagger}$, 
the submanifold $P_n \cap G$ is a totally geodesic subspace. 
Such a group $G$ is called {\em self-adjoint} (or {\em symmetric}), 
and is a {\rm reductive algebraic group}; 
see~\cite[Sections 2.2, 3.1.3, and 3.2]{Wallach_book}.  
Here $P_n \cap G$ is known as a {\em symmetric space} (of nonpositive curvature). 
A particular case we will face is: $G = SL_n$ and $P_n^1 := P_n \cap SL_n = \{x \in P_n \mid \det x =1\}$, 
where the tangent space $T_I(P_n^1)$ at $I$ is given by $\mathfrak{p}_n^1 := \{H \in \mathfrak{p}_n \mid \trace H = 0 \}$.
It is known \cite[II.10.71]{BridsonHaefliger1999} that the boundary $M^{\infty}$ at infinity of $M = P_n \cap G$ 
becomes a {\em spherical building}, and the associated Euclidean cone $CM^{\infty}$ 
becomes a {\em Euclidean building}.
We will consider convex functions 
on these spaces in Section~\ref{sec:applications}.

\subsection{Convex function}
In a Hadamard manifold $M$, by uniqueness of geodesics, 
convexity is naturally introduced. 
A function $f: M \to \RR$ is said to be {\em convex} if for every geodesic $\gamma:[a,b] \to M$ one-dimensional function 
$f\circ \gamma: [a,b] \to \RR$ is convex. 
We will assume twice differentiability for the smoothness of $f$.
Then the convexity condition is equivalent to $(f \circ \gamma)''(t) \geq 0$.
From $(f \circ \gamma)''(t) = (d/dt)\langle \nabla f(\gamma(t)),\dot \gamma(t) \rangle = \langle \nabla_{\dot \gamma(t)} \nabla f(\gamma(t)),\dot \gamma(t) \rangle$,  convexity of $f$ is equivalent to positive semidefiniteness of Hessian $\nabla^2 f(x)$:
\begin{equation*}
\langle \nabla^2 f(x)u,u \rangle \geq 0. 
\end{equation*}
for all $x \in M, u \in T_x$.
We also consider the Lipschitz condition for the gradient vector field $\nabla f$.
For $L \in \RR_+$, 
a function $f:M \to \RR$ is said to be {\em $L$-smooth} if
\[
\langle \nabla^2 f(x)u,u \rangle  \leq L \langle u,u \rangle
\]
for all $x \in M$, $u \in T_x$. That is,  
the operator norm $\| \nabla^2 f(x) \|$ is bounded by $L$ 
for all $x \in M$. 

We next introduce an important tool for studying the unboundedness of convex functions.  
Let us fix $x_0 \in M$. 
The {\em recession function (asymptotic slope)} $f^{\infty} = f_{x_0}^{\infty}: M^{\infty} \to \RR \cup \{\infty\}$~\cite{Hirai_Hadamard2022,KLM2009JDG} 
is defined by
\begin{eqnarray}
f_{x_0}^{\infty}(u) &:=& \lim_{s \to \infty} \frac{f(\exp_{x_0} su) - f(x_0)}{s} = \lim_{s \to \infty}\frac{f(\exp_{x_0} su)}{s}  \nonumber  \\
&=&  \lim_{s \to \infty} \frac{d}{ds} f(\exp_{x_0}su)
\quad (u \in S_{x_0} \simeq M^{\infty}), \label{eqn:recession_function}
\end{eqnarray}
where the limits exist in $\RR \cup \{\infty\}$ due to convexity of $f$ 
(monotonicity of $s \mapsto (f(\exp_{x_0} su) - f(x_0))/s$ and of $s \mapsto (d/ds)  f(\exp_{x_0}su)$) 
and the last equality follows from (\ref{eqn:exercise2}) for $h(t) := (d/dt)f(\exp_{x_0}tu)$.
It is shown~\cite[Lemma 2.10]{KleinerLeeb2006} that if $t \mapsto \exp_{x_0}t u$ 
and $t \mapsto \exp_{y_0}t v$ are asymptotic, then $f_{x_0}^{\infty}(u) = f_{y_0}^{\infty}(v)$.\footnote{
Proof sketch: Let $\alpha(t) := \exp_{x_0}tu$ and
$\beta(t) := \exp_{y_0}tv$, and 
define $u_t \in S_{x_0}$ by $\exp_{x_0}d(x_0,\beta(t)) u_t = \beta(t)$.
By convexity of $f$ along the geodesic between $x_0$ and $\beta(t)$, 
it holds $f(\exp_{x_0} s u_t) - f(x_0) \leq (s/d(x_0,\beta(t))) (f(\beta(t)) - f(x_0))$ for $s \in [0,d(x_0,\beta(t))]$.
By the triangle inequality,
we have $(f(\beta(t)) - f(x_0))/d(x_0,\beta(t)) 
\leq \max_{\sigma \in \{-1,1\}} (f(\beta(t)) - f(x_0))/(t + \sigma d(x_0,y_0)) \to f^{\infty}_{y_0}(v)$ 
for $t \to \infty$. 
By the CAT(0)-inequality on the geodesic triangle of vertices $x_0, \alpha(t), \beta(t)$ and 
by $d(\alpha(t),\beta(t))$ being bounded, 
it holds $\exp_{x_0} s u_t \to \alpha(s)$ for $t \to \infty$. 
Thus we have $f^{\infty}_{y_0}(v) \geq (f(\alpha(s)) - f(x_0))/s \underset{s \to \infty}{\to} f^{\infty}_{x_0}(u)$. 
By symmetry, it holds $f^{\infty}_{x_0}(u) \geq f^{\infty}_{y_0}(v)$, and hence $f^{\infty}_{x_0}(u) = f^{\infty}_{y_0}(v)$.
}
Hence, the recession function 
$f^{\infty}$ is regarded as $M^{\infty} \to \RR \cup \{\infty\}$.
Further, $f^{\infty}$ is naturally extended to 
$CM^{\infty} \to \RR \cup \{\infty\}$ 
by allowing $u$ to any vector in $T_{x_0} \simeq CM^{\infty}$.
If $M = \RR^n$, then $CM^{\infty} = \RR^n$ and $f^{\infty}$ matches the recession function in Euclidean convex analysis; see \cite[Section 8]{Rockafellar} and \cite[Section 3.2]{Hiriart-UrrutyLemarechal}.
As in the Euclidean case,  
the following properties hold:
\begin{eqnarray}
\inf_{\xi \in M^{\infty}}f^{\infty}(\xi) < 0 & \Rightarrow & \inf_{x \in M} f(x) = - \infty. \nonumber \\
\inf_{\xi \in M^{\infty}}f^{\infty}(\xi) > 0
& \Rightarrow & 
\exists x^* \in M:  f(x^*) = \inf_{x \in M} f(x). \label{eqn:exists}
\end{eqnarray}
The second property is included in \cite[Lemma 3.2 (vi)]{KLM2009JDG}.
Moreover,  
it is known \cite{Hirai_Hadamard2022} that $f^{\infty}$ is a positively homogeneous convex function 
on Hadamard space $CM^{\infty}$.

In particular, both $\inf_{\xi \in M^{\infty}}f^{\infty}(\xi) < 0$ and 
$\inf_{x \in M} \|\nabla f(x)\| > 0$ are sufficient conditions for unboundedness of $f$.
In fact, they are equivalent. 
\begin{Prop}[{\cite[Lemma 3.2 (iii), Lemma 3.4]{KLM2009JDG}; see also \cite{Hirai_Hadamard2022}}]\label{prop:f^infty} \  
\begin{itemize}
	\item[(1)]  $\displaystyle \inf_{\xi \in M^{\infty}}f^{\infty}(\xi) < 0$ if and only if 
 $\displaystyle \inf_{x \in M} \|\nabla f(x)\| > 0$.   
%
\item[(2)] If $\displaystyle \inf_{\xi \in M^{\infty}}f^{\infty}(\xi) < 0$, 
then there uniquely exists $\xi^* \in M^{\infty}$ with 
$\displaystyle f^{\infty}(\xi^*) = \inf_{\xi \in M^{\infty}}f^{\infty}(\xi)$.
\end{itemize}
\end{Prop}

The existence in (2) follows from the lower semicontinuity of 
$f^{\infty}$ on the compact space $M^{\infty}$ 
with respect to the cone topology.
The uniqueness of $\xi^*$ in (2) can be seen 
from the positively homogeneous convexity of $f^{\infty}$ on $CM^{\infty}$, as in the Euclidean case.\footnote{If $f^{\infty}(\xi) = f^{\infty}(\xi') = c < 0$,
then by convexity, it holds $f^{\infty}(m) \leq (f^{\infty}(\xi)+ f^{\infty}(\xi'))/2 = c$ for the midpoint $m$ 
of $\xi$ and $\xi'$ in $CM^{\infty}$, and by $\|m\| < 1$ it holds $f^{\infty}(m/\|m\|) = f^{\infty}(m)/\|m\| < c$.}

As a sharpening of the easier part (the only-if part) in (1), 
we here mention the following weak duality relation between 
the gradient norm and the recession function.
\begin{Lem}[Weak duality]\label{lem:weakduality} 
	$\displaystyle 
	\inf_{x \in M}  \|\nabla f(x)\| \geq \sup_{\xi \in B^{\infty}} - f^{\infty}(\xi).
	$
\end{Lem}

\begin{proof}
	For $x\in M$ and $\xi \in B_x \simeq B^\infty$, it holds
	\[
	f^{\infty}(\xi) = \lim_{t \to \infty} \frac{f(\exp_x t\xi)-f(x)}{t} \geq   \lim_{t \to 0} \frac{f(\exp_x t\xi)-f(x)}{t}
	= \langle \nabla f(x),\xi \rangle \geq -  \|\nabla f(x)\|,
	\]
    where the first inequality follows from convexity of $f$ (monotonicity of $t \mapsto (f(\exp_x t\xi)-f(x))/t$) 
    and the last inequality follows from Cauchy-Schwarz and $\|\xi\| \leq 1$.
\end{proof}
In Section~\ref{sec:gradientflow}, 
we show, via the gradient flow of $f$, that  
the equality ({\em strong duality}) always holds.
This technique may be viewed as a refinement of
the proof of the if-part in~\cite[Proposition~\ref{prop:f^infty} (1)]{KLM2009JDG}, 
in which the limit of the normalized gradient flow of $f$ constructs $\xi$ with $f^{\infty}(\xi) < 0$.
A similar gradient-flow approach can be found in 
the setting of GIT~\cite{ChenSun2014,GRS_Book,Woodward2011moment}; see Section~\ref{subsec:norm_minimization}. 
%


\section{Asymptotic behavior of gradient flow}\label{sec:gradientflow}

\subsection{Continuous-time gradient flow}\label{subsec:continuous_GR}
Throughout, $M$ denotes a Hadamard manifold.
Let $f:M \to \RR$ be a twice differentiable convex function.
Consider the following differential equation---the {\em gradient flow} of $f$,
\begin{equation}\label{eqn:GR}
\frac{d x(t)}{dt} = - \nabla f(x(t)), \quad x(0) = x_0.
\end{equation} 
It is clear that the trajectory $x(t)$ is going to minimize $f$; 
see Lemma~\ref{lem:GR-proerties} (2) below. 
In fact, if a minimizer of $f$ exists,
then $x(t)$ converges to a minimizer. 
This convergence is known
for the general setting of Hadamard spaces; see e.g.,  \cite[Theorem 5.1.16]{Bacak_Book2014} and \cite[Theorem 2.41]{Mayer1998}.
Our focus is on the case where $f$ is unbounded below, particularly the case where the minimum gradient-norm is positive.
We establish the following convergence 
of an unbounded gradient flow and strong duality 
between the gradient norm and the recession function. 
\begin{Thm}\label{thm:continuous_main}
	Suppose that $\kappa^* := \inf_{x \in M} \|\nabla f(x)\| > 0$.
	Let $x(t)$ be the solution of (\ref{eqn:GR}). 
 \begin{itemize}
 \item[(1)] $\|\nabla f(x(t))\|$ converges to the minimum gradient-norm $\kappa^*$, and
 \item[(2)] $x(t)$ converges, in cone topology, to the unique minimizer $\xi^*$ of $f^{\infty}$ over $M^{\infty}$,
 \end{itemize}
 where the following equality holds
 \begin{equation}\label{eqn:strongduality}
	\lim_{t \to \infty} \|\nabla f(x(t))\| = \inf_{x \in M}\|\nabla f(x)\| = \sup_{\xi \in M^{\infty}} - f^{\infty}(\xi) = - f^{\infty} \left(\lim_{t \to \infty} x(t)\right).
\end{equation}
 \end{Thm}
 We should mention related results.
In the general setting of Hadamard space $X$, 
Caprace and Lytchak~\cite[Proposition 4.2]{CapraceLytchak2010} showed 
that the gradient-flow curve of a Lipschitz convex function 
with $\kappa^* > 0$ converges to a point in 
the boundary $X^{\infty}$ of $X$.
Their proof relies on a very general result of Karlsson and Margulis~\cite[Theorem 2.1]{KarlssonMargulis1999} for semi-contraction semigroups in uniformly convex spaces. 
Here it is well-known\footnote{
It is found in \cite[Theorem 4.0.4]{Ambrosio_Book} for the general setting of gradient flows 
in metric spaces. 
For our manifold case, it is an easy consequence 
of the first variation formula~\cite[Proposition 2.2]{Sakai1996} as follows: 
$(d/{dt}) d(\phi_t(x),\phi_t(y))^2/2 = \langle - \nabla f(\phi_t(y)), \dot \gamma(1)\rangle - \langle - \nabla f(\phi_t(x)), \dot \gamma(0)\rangle = - \int_{0}^1 (d/dt) \langle \nabla f(\gamma(s)), \dot \gamma(s)\rangle ds = - \int_{0}^1 \langle \nabla^2 f(\gamma(s))\dot \gamma(s), \dot \gamma(s)\rangle ds \leq 0$, where $\gamma:[0,1] \to M$ is a geodesic from $\phi_t(x)$ to $\phi_t(y)$.
}
that the gradient-flow semigroup $\phi_t$ satisfies the {\em \mbox{(semi-)}contraction property}: 
\begin{equation}\label{eqn:contraction}
d(\phi_t(x),\phi_t(y)) \leq d(x,y) \quad (t \in \RR_+, x,y \in M),
\end{equation}
where $\phi_t(x)$ 
is the solution of (\ref{eqn:GR}) with initial point $x(0) = x$.
If the {\em velocity of escape}
\begin{equation}\label{eqn:velocity}
\kappa^*(x) :=  \limsup_{t \to \infty} \frac{d(\phi_t(x),x)}{t}
\end{equation}
is positive, then \cite[Theorem 2.1]{KarlssonMargulis1999} is applicable 
for convergence of $\phi_t(x)$ in $M^{\infty}$;   
Caprace and Lytchak  actually showed that $\kappa^* > 0$ implies $\kappa^*(x) > 0$.
Although one can deduce the entire statement of Theorem~\ref{thm:continuous_main} from this with more effort, we take a different approach that relies neither on~\cite{KarlssonMargulis1999} nor on the contraction property~(\ref{eqn:contraction}).
As mentioned after Lemma~\ref{lem:weakduality}, 
our proof is partly inspired by an idea in \cite{KLM2009JDG}, but it 
directly establishes the relation (\ref{eqn:strongduality}).
An advantage of this approach is that it can adapt to 
the discrete setting in Section~\ref{subsec:GR_D}. 

We start with the following well-known properties of gradient flows.
\begin{Lem}\label{lem:GR-proerties}
	\begin{itemize}
		\item[(1)] The solution $x(t)$ of (\ref{eqn:GR}) is defined on $\RR_+$.
		\item[(2)] $t \mapsto f(x(t))$ is nonincreasing.
		\item[(3)] $t \mapsto \|\nabla f(x(t))\|$ is nonincreasing.
	\end{itemize}
\end{Lem}
We describe a proof since the intermediate equations will be used. 

\begin{proof}
	(2) follows from
	$
	(d/dt)f(x(t)) = \langle \nabla f(x(t)),\dot x(t) \rangle = - \| \nabla f(x(t))\|^2 
	\leq  0. 
	$
 
	(3) follows from
 $(d/dt) \| \nabla f(x(t))\|^2 =  - 2 \langle \nabla^2 f(x(t)) \dot x(t),\dot x(t) \rangle \leq 0$ by convexity of $f$ (positive semidefiniteness of $\nabla^2 f(x(t))$).
	
	(1). Suppose that $x(t)$ is defined on $[0,T)$ for finite $T> 0$.
    For $0 \leq t \leq t' < T$, it holds
    \[
    d(x(t),x(t')) \leq \int_t^{t'} \|\dot x(s)\| ds  \leq \|\nabla f(x_0)\| (t' - t),    
    \]
    where the second inequality follows from (3). 
    Therefore, $x(t)$ is Cauchy for $t \to T$. 
    Since $M$ is complete, the limit $x^*:=\lim_{t \to T}x(x)$ exists in $M$.
    Then $x(t)$ is connected to the solution of $\dot y(t) = - \nabla f(y(t))$, $y(0) = x^*$, 
    and is defined on $[0, T+ \epsilon)$ for some $\epsilon > 0$. 
    If we take maximal $T$, it must be $T = \infty$.   
\end{proof}

\begin{proof}[Proof of Theorem~\ref{thm:continuous_main}]
Let $\kappa := \lim_{t \to \infty} \|\nabla f(x(t))\| \geq \kappa^* > 0$.
	First, we note
	\begin{eqnarray}
	&& f(x(t)) - f(x_0) = \int_{0}^{t} \frac{d}{d\tau} f(x(\tau)) d \tau=  - \int_{0}^t \|\nabla f(x(\tau))\|^2 d\tau \leq - \kappa^2t, \label{eqn:note1} \\
	&& d(x(t),x_0)  \leq   \int_0^t \|\dot x(\tau)\| d\tau =  \int_0^t\|\nabla f(x(\tau))\|d\tau.\label{eqn:note2}
	\end{eqnarray}
    where the last inequality in (\ref{eqn:note1}) follows from Lemma~\ref{lem:GR-proerties}~(3).
	  Then it holds $d(x(t),x_0) \to \infty$ for $t \to \infty$. 
   Otherwise, $x(t)$  
   has an accumulation point $x^*$ in $M$ and $f(x^*) = - \infty$ by (\ref{eqn:note1}), contradicting $f(x^*) \in \RR$.
   
	Define $u(t) \in S_{x_0}$  
	via 
	$
	x(t) = \exp_{x_0} d(x(t),x_0) u(t).
	$
	For $s \in (0, d(x(t),x_0)]$, 
	by convexity of $f$ along the geodesic from $x_0$ to $x(t)$, it holds
	\begin{equation*}\label{eqn:convexity}
	f(\exp_{x_0} s u(t)) - f(x_0) \leq \frac{s}{d(x(t),x_0)} (f(x(t))- f(x_0)). 
	\end{equation*}
From this, we have
\begin{eqnarray*}
	&& \frac{f(\exp_{x_0} s u(t)) - f(x_0)}{s}
	\leq \frac{f(x(t))- f(x_0)}{d(x(t),x_0)}
	\leq - \frac{\int_0^t \|\nabla f(x(\tau))\|^2 d\tau}{\int_0^t \|\nabla f(x(\tau))\| d\tau}\\\notag
	&& \leq  -\frac1t\int_0^t\|\nabla f(x(\tau))\| d\tau \leq  - \kappa,
	\end{eqnarray*}
 where the second inequality follows from (\ref{eqn:note1}) and (\ref{eqn:note2}), 
	the third from the Cauchy--Schwartz inequality 
 $(\int_0^t F(\tau)G(\tau) d\tau)^2 \leq \int_0^t F(\tau)^2 d\tau \int_0^t  G(\tau)^2 d\tau$ 
 for $F(\tau) := \|\nabla f(x(\tau))\|$ and $G(\tau) := 1$, and the fourth from Lemma~\ref{lem:GR-proerties}~(3).

	Choose any convergence subsequence $u(t_i)$ with $t_i \to \infty$ $(d(x(t_i),x_0) \to \infty)$ and $u(t_i) \to u^*$.
	Then it holds
	\[
	\frac{f(\exp_{x_0} s u^*) - f(x_0)}{s} \leq  - \kappa.
	\]
	For $s \to \infty$, we have
	$
	f^{\infty}(u^*) \leq - \kappa.
	$
	Then, we have 
	\[
	\inf_{\xi \in M^{\infty}} f^{\infty}(\xi) \leq f^{\infty}(u^*) \leq -  \kappa \leq - \kappa^* = \sup_{x \in M} - \|\nabla f(x)\| \leq \inf_{\xi \in M^{\infty}} f^{\infty}(\xi),
	\]
    where we use the weak duality (Lemma~\ref{lem:weakduality}) for the last inequality.
    This shows $\kappa = \kappa^*$ and proves (\ref{eqn:strongduality}).
    Since the minimizer $\xi^*$ of $f^{\infty}$ over $M^\infty$ uniquely exists (Proposition~\ref{prop:f^infty}~(2)), 
	it must hold $\xi^* = u^*$. 
	We showed that any convergent subsequence $u(t_i)$ of $u(t)$ 
	converges to $\xi^*$. Since $S_{x_0}$ is compact,  
	$u(t)$ itself converges to $\xi^*$.
\end{proof}

Even if $\kappa^* = 0$, the strong duality holds (since $f^{\infty}(0) = 0$).
\begin{Cor}\label{cor:duality}
	$\displaystyle
	\inf_{x \in M}\|\nabla f(x)\| = \sup_{\xi \in B^{\infty}} - f^{\infty}(\xi).
	$
\end{Cor}

The velocity of escape (\ref{eqn:velocity}) coincides with the minimum gradient-norm.
\begin{Prop}\label{prop:velocity_of_escape}
Suppose that $\kappa^* :=  \inf_{x \in M} \|\nabla f(x)\| > 0$. Let
$\xi^* \in S_{x_0}$ denote the representative of the unique minimizer of $f^{\infty}$ over $M^{\infty} \simeq S_{x_0}$.
Then the following hold:
\begin{itemize}
\item[(1)] $\displaystyle \lim_{t \to \infty} \frac{d(x_0,x(t))}{t} = \kappa^*$. 
\item[(2)] $\displaystyle \lim_{t \to \infty} \frac{\exp_{x_0}^{-1} x(t)}{t} = \kappa^*\xi^*$. 
\end{itemize}
\end{Prop}
\begin{proof}
(1). For $t > s \geq 0$, it holds
$d(x(s),x(t)) \leq \int_s^t\|\nabla f(x(\tau))\|d\tau \leq \|\nabla f(x(s))\|(t-s)$ (by Lemma~\ref{lem:GR-proerties}~(3)).
Hence
\begin{equation}\label{eqn:limsup<=kappa*}
\limsup_{t \to \infty} \frac{d(x_0,x(t))}{t} = \limsup_{t \to \infty}  \frac{d(x(s),x(t))}{t-s} \leq \|\nabla f(x(s))\| \quad \underset{s \to \infty}{\longrightarrow} \kappa^*,
\end{equation}
where the convergence of $\|\nabla f(x(s))\|$ to $\kappa^*$ 
follows from Theorem~\ref{thm:continuous_main}~(1). 
On the other hand, by taking the unit speed geodesic $\gamma$ from $x(s)$ to $x(t)$, we have
\begin{eqnarray*}
&& - \|\nabla f(x(t))\|^2(t-s) \geq -\int_s^t \|\nabla f(x(\tau))\|^2 d\tau = f(x(t)) - f(x(s)) \\
&& \geq \langle \dot \gamma(0),\nabla f(x(s)) \rangle d(x(s),x(t)) \geq - \| \nabla f(x(s))\| d(x(s),x(t)), 
\end{eqnarray*}
where the first equality follows from Lemma~\ref{lem:GR-proerties}~(3) and the second inequality 
from convexity of $f$ along $\gamma$, and the last from the Cauchy-Schwarz inequality. 
Thus it holds
\begin{eqnarray}
&& \liminf_{t \to \infty} \frac{d(x_0,x(t))}{t} 
\geq \liminf_{t \to \infty} \frac{d(x(t),x(s)) - d(x_0,x(s))}{t} =  \liminf_{t \to \infty} \frac{d(x(t),x(s))}{t-s} \nonumber \\
&& \geq \frac{\lim_{t \to \infty}\|\nabla f(x(t))\|^2}{\|\nabla f(x(s))\|} =  \frac{(\kappa^*)^2}{\|\nabla f(x(s))\|} \quad \underset{s \to \infty}{\longrightarrow} \kappa^*.  \label{eqn:liminf=>kappa*}
\end{eqnarray}
By (\ref{eqn:limsup<=kappa*}) and (\ref{eqn:liminf=>kappa*}), we have
\[
\kappa^* \leq \liminf_{t \to \infty} \frac{d(x_0,x(t))}{t} \leq \limsup_{t \to \infty} \frac{d(x_0,x(t))}{t} \leq \kappa^*. 
\]
(2). By Theorem~\ref{thm:continuous_main}, it holds
$\lim_{t \to \infty} \frac{\exp_{x_0}^{-1} x(t)}{d(x_0,x(t))} = \xi^*$. Therefore, by (1), we have
\[
\lim_{t \to \infty}\frac{\exp_{x_0}^{-1} x(t)}{t} = \lim_{t \to \infty} \frac{\exp_{x_0}^{-1} x(t)}{d(x_0,x(t))} \frac{d(x_0,x(t))}{t}   = \kappa^* \xi^*. \qedhere
\]
\end{proof}
We next consider ``convergence" of the gradient 
$\nabla f(x(t))$. 
Since the space $T_{x(t)}$ varies, the convergence concept of $\nabla f(x(t))$ is less obvious.
In our intuition, $\nabla f(x(t))$ and $\xi^*$ 
would have opposite directions in the limit. The following partially justifies this intuition.
\begin{Prop}\label{prop:accumulation_of_transported_gradient}
Suppose that $\kappa^* := \inf_{x \in M} \|\nabla f(x)\| > 0$.
Let $\xi^* \in S_{x_0}$ denote the representative of the unique minimizer of $f^{\infty}$ over $M^{\infty} \simeq S_{x_0}$. Then it holds
\begin{equation*}
   \liminf_{t \to \infty} \|\tau_{x(t) \to x_0} \nabla f(x(t)) + \kappa^* \xi^*\|  = 0.
\end{equation*}
\end{Prop}
\begin{Question}\label{que:convergence}
    Does 
    $\displaystyle \lim_{t \to \infty} \tau_{x(t) \to x_0} \nabla f(x(t)) = -  \kappa^* \xi^*$
    hold ? 
\end{Question}
We will see in Section~\ref{sec:applications} 
that this property has important consequences.

\begin{proof}[Proof of Proposition~\ref{prop:accumulation_of_transported_gradient}]
Let $\gamma_t$ be the unit-speed geodesic from $x_0$ to $x(t)$.
Let $d(t) := d(x_0,x(t))$.
Then, by \cite[Chapter III, Proposition 4.8 (1)]{Sakai1996}, it holds
$
     d(t)' = \left\langle\dot\gamma_t(d(t)), \dot{x}(t)\right\rangle.
$
Therefore,  we have
\begin{equation}\label{eqn:d(t)'}
        \limsup_{t \to \infty} d(t)' = \limsup_{t \to \infty} \left\langle\dot\gamma_t(d(t)), \dot{x}(t)\right\rangle \le \lim_{t \to \infty} \|\dot{x}(t)\| = \lim_{t \to \infty} \|\nabla f(x(t))\| = \kappa^\ast.
\end{equation}
On the other hand, by Proposition~\ref{prop:velocity_of_escape}, it holds $\kappa^* = \limsup_{t \to \infty} d(t)/t \leq \limsup_{t \to \infty} d(t)'$, 
where the inequality follows from (\ref{eqn:exercise2}) with $h(t) := d'(t)$.
Thus, the equality holds in (\ref{eqn:d(t)'}).
Necessarily we have
\begin{equation}
\limsup_{t\to \infty} \angle(\dot \gamma_t(d(t)), \nabla f(x(t)))  = \pi.
\end{equation}
By $\|\nabla f(x(t))\| \to \kappa^*$, 
we have 
$
\liminf_{t\to \infty} \| \nabla f(x(t)) + \kappa^* \dot \gamma_t(d(t)) \| = 0$.
With parallel transport $\tau_{x(t) \to x_0}$ and 
$\dot \gamma_t(0) \to \xi^*$, we have the claim.
\end{proof}

\subsection{Discrete-time gradient flow (gradient descent)}\label{subsec:GR_D}

Next we consider the discrete version.
Suppose that $f: M \to \RR$ is an  $L$-smooth convex function.
Consider the following sequence: 
\begin{equation}\label{eqn:DGR}
x_{i+1} := \exp_{x_i} \left( - \frac{1}{L}\nabla f(x_i) \right) \quad  (i=0,1,\ldots).
\end{equation}
This is nothing but the trajectory generated by gradient descent with initial point $x_0$ and step-size $1/L$; we discuss in Remark~\ref{rem:proximal} 
another type of discrete gradient flow.
The convergence/accumulation of $x_i$ to a minimizer of $f$ 
can be shown under several reasonable assumptions; see e.g., \cite[Theorem 11.29]{Boumal_Book}.
For the unbounded case, as in the continuous setting, we establish the following. 
\begin{Thm}\label{thm:discrete_main}
Suppose that $\kappa^* := \inf_{x \in M} \|\nabla f(x)\| > 0$.
	Let $x_i$ be the sequence in (\ref{eqn:DGR}). 
 \begin{itemize}
 \item[(1)] $\|\nabla f(x_i)\|$ converges to the minimum gradient-norm $\kappa^*$, and
 \item[(2)] $x_i$ converges, in cone topology, to the unique minimizer $\xi^* \in M^{\infty}$ of $f^{\infty}$.
 \end{itemize}
 Hence, the following holds
	\begin{equation}
	\lim_{i \to \infty} \|\nabla f(x_i)\| = \inf_{x \in M}\|\nabla f(x)\| = \sup_{\xi \in M^{\infty}} - f^{\infty}(\xi) = - f^{\infty} \left(\lim_{i \to \infty}x_i\right).
	\end{equation}
\end{Thm}
Our original attempt proving this was to 
establish the contraction property
\begin{equation}\label{eqn:contraction_D}
d(\phi_i(x),\phi_i(y)) \leq d(x,y) \quad (x,y \in M, i=1,2,\ldots),
\end{equation}
for the semigroup $\phi_i$ of (\ref{eqn:DGR}), 
and to apply the approach of~\cite{CapraceLytchak2010, KarlssonMargulis1999}. 
However, we were unable to do so, and we do not know whether (\ref{eqn:contraction_D}) is true.
Note that (\ref{eqn:contraction_D}) is true in Euclidean space $M = \RR^n$; 
see e.g., \cite[Example 1]{SanzSernaZygalakis2020}.

The proof goes a way analogous to Theorem~\ref{thm:continuous_main}.
Corresponding to Lemma~\ref{lem:GR-proerties}, the following properties hold.
\begin{Lem}\label{lem:f(x_i)}
	\begin{itemize}
		\item[(1)] $\displaystyle f(x_{i+1}) \leq f(x_i) - \frac{1}{L} \| \nabla f(x_{i+1})\|^2$.
		\item[(2)] $\|\nabla f(x_{i+1})\| \leq \|\nabla f(x_{i})\|$.
	\end{itemize}
\end{Lem}
Contrary to the well-known inequality 
$f(x_{i+1}) \leq f(x_i) - (1/2L) \| \nabla f(x_{i})\|^2$ (see \cite[(11.15)]{Boumal_Book}),  
our inequality (1) seems less well-known; 
see Remark~\ref{rem:descend} for further discussion.

\begin{proof}
	(2). Let $\gamma(t) := \exp_{x_i} -t \nabla f(x_i)$. Then we have
	\begin{eqnarray}
		\tau^{-1/L}_\gamma \nabla f(x_{i+1})  &=& \nabla f(x_i) + \int_0^{1/L} \frac{d}{ds} \tau^{-s}_\gamma \nabla f(\gamma(s)) ds 
		\nonumber \\
		&=& \nabla f(x_i) + \int_0^{1/L} \tau^{-s}_\gamma \nabla_{\dot \gamma(s)}\nabla f(\gamma(s)) ds \nonumber \\
		&=& \nabla f(x_i) + \int_0^{1/L} \tau^{-s}_\gamma \nabla^2 f(\gamma(s)) \dot \gamma(s) ds \nonumber \\
		&=& L \int_0^{1/L} \tau^{-s}_\gamma \left(I -  \frac{1}{L} \nabla^2 f(\gamma(s)) \right) \tau^s_\gamma  \nabla f(x_i) ds, \label{eqn:nablaf(x_i+1)}
	\end{eqnarray}
	where we use the definition (\ref{eqn:Hessian}) of $\nabla^2$
    and 
    $\dot \gamma(s) = \tau^s_\gamma \dot \gamma(0) = - \tau^s_\gamma \nabla f(x_i)$ as $\gamma$ is a geodesic.
 Since $\langle, \rangle$ is invariant under parallel transport, 
 the operator norm of $\tau^{-s}_\gamma (I - (1/L)\nabla^2 f(\gamma(s)) )\tau^s_\gamma$ is equal to that of $I - (1/L)\nabla^2 f(\gamma(s))$.
	By convexity and $L$-smoothness, all eigenvalues of $\nabla^2 f(\gamma(s))$ belong to $[0,L]$.
	Hence we have
	\begin{eqnarray*}
		 \|\nabla f(x_{i+1})\| &=& \|\tau^{-1/L}_\gamma \nabla f(x_{i+1})\| \leq  
		L \int_0^{1/L} \left\|  I - \frac{1}{L} \nabla^2 f(\gamma(s)) \right\| \|\nabla f(x_i)\| ds\\
		& \leq & \|\nabla f(x_i)\|,
	\end{eqnarray*}
    which proves (2).
    
	We now prove (1). From (\ref{eqn:nablaf(x_i+1)}), we have
	\begin{eqnarray*}
		 &&\left\| \tau^{-1/L}_\gamma \nabla f(x_{i+1}) - \frac{1}{2} \nabla f(x_i) \right\|  =  L \left\| \int_0^{1/L} \tau^{-s}_\gamma \left(\frac{1}{2}I -  \frac{1}{L} \nabla^2 f(\gamma(s)) \right) \tau^s_\gamma  \nabla f(x_i) ds \right\| \\
		&& \leq L  \int_0^{1/L} \left\| \left(\frac{1}{2}I -  \frac{1}{L} \nabla^2 f(\gamma(s)) \right)\right\|  \left\| \nabla f(x_i) \right\| ds  \leq \frac{1}{2} \left\| \nabla f(x_i) \right\|. 
	\end{eqnarray*}
    By squaring this and applying the rearrangement $\|a-b\|^2\leq \|b\|^2 \Rightarrow \|a\|^2 \leq 2 \langle a,b\rangle$,
	we have $\|\tau^{-1/L}_\gamma \nabla f(x_{i+1})\|^2 \leq \langle \tau^{-1/L}_\gamma \nabla f(x_{i+1}), \nabla f(x_i) \rangle$, particularly, 
	\begin{equation}
        \label{eqn:next_gradient}
    	\|\nabla f(x_{i+1}) \|^2  \leq \langle \nabla f(x_{i+1}), \tau^{1/L}_\gamma \nabla f(x_i) \rangle. 
	\end{equation}
	From convexity, it holds
	\begin{eqnarray*}
		&& f(x_i) \geq  f(x_{i+1}) + \frac{1}{L} \frac{d}{dt}  f(\gamma(1/L - t))\mid_{t=0} 
		=  f(x_{i+1}) - \frac{1}{L} \langle \nabla f(x_{i+1}), \dot \gamma(1/L) \rangle \\
		&& =  f(x_{i+1}) + \frac{1}{L} \langle \nabla f(x_{i+1}), \tau^{1/L}_\gamma \nabla f(x_i) \rangle \geq f(x_{i+1}) + \frac{1}{L} \|\nabla f(x_{i+1})\|^2,
	\end{eqnarray*}
    where we use (\ref{eqn:next_gradient}) for the last inequality.
\end{proof}

\begin{proof}[Proof of Theorem~\ref{thm:discrete_main}]
The proof is similar to that of Theorem~\ref{thm:continuous_main}.
Let $\kappa := \lim_{i \to \infty} \|\nabla f(x_i)\| \geq \kappa^*$.
	For $i > 0$, we have
	\begin{eqnarray}
	f(x_i) - f(x_0) &\le& -\frac1L\sum_{k = 1}^i\|\nabla f(x_k)\|^2 \leq - \frac{i}{L} \kappa, \label{eqn:d_note1}\\
 d(x_i, x_0) &\le& \sum_{k=0}^{i-1} d(x_k,x_{k+1}) = \frac1L\sum_{k = 0}^{i-1}\|\nabla f(x_k)\|, \label{eqn:d_note2}
	\end{eqnarray}
 where (\ref{eqn:d_note1}) follows from Lemma~\ref{lem:f(x_i)} 
 and (\ref{eqn:d_note2}) follows from
	the triangle inequality and $d(x,\exp_x u) = \|u\|$ with (\ref{eqn:DGR}). 
	Then $d(x_i, x_0) \to \infty$ is shown as in the proof of Theorem~\ref{thm:continuous_main}.
 
	Let $u_{i} \in S_{x_0}$ be defined via 
	$
	x_i = \exp_{x_0} d(x_i,x_0) u_i.
	$
	For $s \in (0, d(x_i, x_0)]$, by convexity of $f$ along geodesic $s \mapsto \exp_{x_0}s u_i$, it holds
	\begin{equation*}\label{eqn:u_i}
	f(\exp_{x_0} s u_i) - f(x_0)
	\le \frac{s}{d(x_i, x_0)}(f(x_i) - f(x_0)).
	\end{equation*}
	  From this, we have
	\begin{eqnarray}
	&& \frac{f(\exp_{x_0} s u_i) - f(x_0)}{s}
	\leq \frac{f(x_i) - f(x_0)}{d(x_i, x_0)}
	\leq \frac{-\sum_{k = 1}^i\|\nabla f(x_k)\|^2}{d(x_i, x_0)}
    \leq \frac{-\sum_{k = 1}^i\|\nabla f(x_k)\|^2}{\sum_{k = 0}^{i-1}\|\nabla f(x_k)\|}
    \notag\\
	&&= -\frac{\sum_{k = 0}^{i - 1}\|\nabla f(x_k)\|^2}{\sum_{k = 0}^{i-1}\|\nabla f(x_k)\|} + \frac{\|\nabla f(x_0)\|^2 - \|\nabla f(x_i)\|^2}{\sum_{k = 0}^{i-1}\|\nabla f(x_k)\|}\notag\\
	&&\leq -\frac1i\sum_{k = 0}^{i-1}\|\nabla f(x_k)\| + \frac{\|\nabla f(x_0)\|^2}{\sum_{k = 0}^{i-1}\|\nabla f(x_k)\|} \leq - \kappa + \frac{1}{i} \frac{\|\nabla f(x_0)\|^2}{\kappa}, \label{eqn:second_term}
	\end{eqnarray}
where the second inequality follows from (\ref{eqn:d_note1}), the third from (\ref{eqn:d_note2}) 
and the negativity of the numerator, 
the fourth from the Cauchy-Schwarz inequality
$(\sum_k F_kG_k)^2 \leq \sum_k F_k^2 \sum_k G_k^2$, and the fifth from Lemma~\ref{lem:f(x_i)}~(2). 

Choose any convergent subsequence $\{u_{i_k}\}$ of $\{u_i\}$, 
	which converges to $u^* \in S_{x_0}$.
 The second term of (\ref{eqn:second_term}) vanishes as $i_k \to \infty$. Then it holds
	\begin{equation*}
	\frac{f(\exp_{x_0} s u^{\ast}) - f(x_0)}{s} \le -\kappa.
	\end{equation*}
	By $s \to \infty$, we have
	$
	f^{\infty}(u^*) \le -\kappa.
	$
	The rest is the same as the last part of the proof of Theorem~\ref{thm:continuous_main}.
\end{proof}

We note the limiting behavior of the decrement of $f(x_i)$ and the change of $\nabla f(x_i)$. 
\begin{Lem}\label{lem:f(x_i+1)-f(x_i)}
\begin{itemize}
        \item[(1)] $\displaystyle \lim_{i \to \infty} f(x_{i+1})- f(x_i) = - \frac{(\kappa^*)^2}{L}$.
\item[(2)] $\displaystyle \lim_{i \to \infty} \| \tau_{x_{i} \to x_{i+1}}\nabla f(x_{i}) - \nabla f(x_{i+1})\| = 0$.
\end{itemize}
\end{Lem}

\begin{proof}
(1). By convexity and Lemma~\ref{lem:f(x_i)}~(1), we have
\[
- \frac{1}{L}\|\nabla f(x_i)\|^2 \leq f(x_{i+1}) -f(x_i) \leq - \frac{1}{L}\|\nabla f(x_{i+1})\|^2. 
\]
By $i \to \infty$ with Theorem~\ref{thm:discrete_main}, we have the claim. 

(2). The inequality (\ref{eqn:next_gradient}) is also written as
   \[
        \|\nabla f(x_{i + 1})\|^2 \le \|\nabla f(x_i)\|\|\nabla f(x_{i + 1})\|\cos \angle(\nabla f(x_{i + 1}), \tau_{x_i \to x_{i+1}}\nabla f(x_i)).
    \]
    By $\|\nabla f(x_i)\| \to \kappa^\ast$, we have 
    $\angle(\nabla f(x_{i + 1}), \tau_{x_i \to x_{i+1}}\nabla f(x_i)) \to 0$, and the claim follows.
\end{proof}

The discrete version of Proposition~\ref{prop:velocity_of_escape} is the following.
\begin{Prop}\label{prop:velocity_of_escape_D}
Suppose that $\kappa^* := \inf_{x \in M}\|\nabla f(x)\| > 0$.
Let $\xi^* \in S_{x_0}$ denote the representative of the unique minimizer of $f^{\infty}$ over $M^{\infty} \simeq S_{x_0}$.
\begin{itemize}
\item[(1)] $\displaystyle \lim_{i \to \infty}\frac{d(x_0,x_i)}{i} = \frac{\kappa^*}{L}$.
\item[(2)] $\displaystyle \lim_{i \to \infty}\frac{\exp_{x_0}^{-1}x_i}{i} = \frac{\kappa^*\xi^*}{L}$.
\end{itemize}
\end{Prop}

\begin{proof}
(1). As in (\ref{eqn:d_note2}), it holds
$d(x_0, x_i) \leq \sum_{k = 0}^{i - 1}d(x_k, x_{k + 1}) = \sum_{k = 0}^{i - 1}\|\nabla f(x_k)\|/L$.
Hence, with (\ref{eqn:exercise1}) for $a_i := \|\nabla f(x_i)\|$, we have
\begin{equation}\label{eqn:limsup<=kappa*/L}
\limsup_{i \to \infty} \frac{d(x_0, x_i)}{i} \leq  \frac{1}{L} \limsup_{i \to \infty} \frac{1}{i}\sum_{k = 0}^{i - 1}\|\nabla f(x_k)\| \leq \frac{1}{L}\limsup_{k \to \infty} \|\nabla f(x_k)\| = \frac{\kappa^\ast}{L}.
\end{equation}
On the other hand, for arbitrary $0 \leq i < j$, we have
\begin{eqnarray*}
&& -\frac{j - i}{L}\|\nabla f(x_j)\|^2 \geq -\frac{1}{L}\sum_{k = i + 1}^j\|\nabla f(x_k)\|^2 \geq (f(x_j) - f(x_i))\\
&& \geq \langle\dot\gamma(0), \nabla f(x_i)\rangle d(x_i, x_j) \geq -\|\nabla f(x_i)\|d(x_i, x_j),    
\end{eqnarray*}
where the first inequality follows from Lemma~\ref{lem:f(x_i)}~(2), the second from Lemma~\ref{lem:f(x_i)}~(1), 
and the third from the convexity of $f$ along unit-speed geodesic $\gamma$ from $x_i$ to $x_j$.
Thus, for arbitrary $i \ge 0$, it holds
\begin{eqnarray}
&& \liminf_{j \to \infty}\frac{d(x_0, x_j)}{j} \geq \liminf_{j \to \infty}\frac{d(x_i, x_j) - d(x_0, x_i)}{j} = \liminf_{j \to \infty}\frac{d(x_i, x_j)}{j - i} \nonumber \\
&& \ge \frac{1}{L}\liminf_{j \to \infty}\frac{\|\nabla f(x_j)\|^2}{\|\nabla f(x_i)\|} = \frac{1}{L}\frac{(\kappa^\ast)^2}{\|\nabla f(x_i)\|} \quad \underset{i \to \infty}{\longrightarrow} \frac{\kappa^\ast}{L}. \label{eqn:liminf=>kappa*/L}
\end{eqnarray}
By (\ref{eqn:limsup<=kappa*/L}) and (\ref{eqn:liminf=>kappa*/L}), we have
\[
\frac{\kappa^*}{L} \leq \liminf_{i \to \infty} \frac{d(x_0,x_i)}{i} \leq \limsup_{i \to \infty} \frac{d(x_0,x_i)}{i} \leq \frac{\kappa^*}{L}.
\]
(2). As in the proof of Proposition~\ref{prop:velocity_of_escape}~(2), by Theorem~\ref{thm:discrete_main} and  the above (1), we have
\[
\lim_{i \to \infty}\frac{\exp_{x_0}^{-1}x_i}{i} = \lim_{i \to \infty}\frac{\exp_{x_0}^{-1}x_i}{d(x_0,x_i)} \frac{d(x_0,x_i)} {i} = \frac{\kappa^*\xi^*}{L}. \qedhere 
\]
\end{proof}

For convergence of $\nabla f(x_i)$, the same property of Proposition~\ref{prop:accumulation_of_transported_gradient} holds: 
\begin{Prop}\label{prop:accumulation_of_transported_gradient_D}
Suppose that $\kappa^* := \inf_{x \in M} \|\nabla f(x)\| > 0$.
Let $\xi^* \in S_{x_0}$ denote the representative of the unique minimizer of $f^{\infty}$ over $M^{\infty} \simeq S_{x_0}$. Then it holds
\begin{equation*}
   \liminf_{i \to \infty} \|\tau_{x_i \to x_0} \nabla f(x_i) + \kappa^* \xi^*\|  = 0.
\end{equation*}
\end{Prop}
\begin{Question}\label{que:convergence_D}
    Does $\lim_{i \to \infty} \tau_{x_i \to x_0} \nabla f(x_i) = -  \kappa^* \xi^*$
    hold ?
\end{Question}
\begin{proof}[Proof of Proposition~\ref{prop:accumulation_of_transported_gradient_D}]
Let $d_i := d(x_0,x_i)$. We first show 
\begin{equation}\label{eqn:d_i+1-d_i}
 \limsup_{i \to \infty} d_{i + 1} - d_i = \kappa^\ast/L. 
\end{equation}
Indeed, 
by the triangle inequality and Theorem~\ref{thm:discrete_main}~(1), we have
    $
        \limsup_{i \to \infty} d_{i + 1} - d_i \le \limsup_{i \to \infty} d(x_i, x_{i + 1}) = \limsup_{i \to \infty} \|\nabla f(x_i)\|/L = \kappa^\ast/L
    $.
 On the other hand, 
    by Proposition \ref{prop:velocity_of_escape_D}~(3), it holds
    $\kappa^*/L= \limsup_{i \to \infty}  d_i/i \leq \limsup_{i \to \infty} d_{i + 1} - d_i$, 
    where the inequality follows from (\ref{eqn:exercise1}) for $a_i := d_{i+1}-d_i$.

    Consider the geodesic triangle of vertices $x_0,x_{i-1}, x_i$. 
    Let $\gamma_i$ denote the unit-speed geodesic from $x_0$ to $x_i$.
    Let $\theta_i$ denote the angle at vertex $x_i$ of this triangle. Then 
    \[
    \theta_i = \angle (\dot \gamma_i(d_i), -\tau_{x_{i-1}\to x_i} \nabla f(x_{i-1})). 
    \]
   By the law of cosines in CAT(0) space $M$ (see e.g., \cite[II.1.9 (2)]{BridsonHaefliger1999}), we have
      \begin{equation*}
        \cos  \theta_i 
        \ge \frac{d_i^2 + d(x_{i - 1}, x_i)^2 - d_{i - 1}^2}{2d_id(x_{i - 1}, x_i)}
        = \frac{d(x_{i - 1}, x_i)}{2d_i} + \frac12\left(1 + \frac{d_{i - 1}}{d_i}\right)\frac{d_i - d_{i - 1}}{d(x_{i - 1}, x_i)}.
    \end{equation*}
    Take $\limsup_{i \to \infty}$ in this inequality. 
    By $d_i =d(x_0,x_i) \to \infty$, $d(x_{i-1},x_i) = \|\nabla f (x_{i-1})\|/L \to \kappa^*/L$ (from Theorem~\ref{thm:discrete_main}~(1)), 
    $d_{i-1}/d_i \to 1$ (seen from Proposition~\ref{prop:velocity_of_escape_D}~(1)), and (\ref{eqn:d_i+1-d_i}), we have
    $
    \limsup_{i \to \infty} \cos \theta_i \geq 1
    $, and 
    $
    \liminf_{i \to \infty} \theta_i = 0.
    $
    By Lemma~\ref{lem:f(x_i+1)-f(x_i)} (2), 
    it holds $\angle (\nabla f(x_{i}), \tau_{x_{i-1} \to x_i} \nabla f(x_{i-1})) \to 0$ and
     \[
    \limsup_{i \to \infty} \angle (\dot \gamma_i(d_i), \nabla f(x_{i})) = \pi. 
    \]
    By taking parallel transport $\tau_{x_i \to x_0}$ and $\dot \gamma_i(0) \to \xi^*$, we have the claim.
\end{proof}

\begin{Rem}\label{rem:proximal}
Another type of discrete gradient flow, well-studied in the literature of nonpositively-curved space (see \cite{Bacak_Book2014,Mayer1998,OhtaPalfia2015}), 
is defined via the {\em resolvent map} 
$J_\lambda^f:M \to M$, 
\begin{equation}\label{eqn:resolvent}
J_\lambda^f(x) := \argmin_{y \in M} f(y) + \frac{1}{2\lambda}d(x,y)^2 \quad (x \in M), 
\end{equation}
where $\lambda$ is a positive parameter.
Let $\lambda_i$ be a sequence of positive reals 
(satisfying $\lambda_i \to 0$ and $\sum_i \lambda_i \to \infty$). 
Then a discrete analogue ({\em proximal point method}) of gradient flow is as follows: 
\begin{equation}\label{eqn:proximal}
x_{i+1} = J_{\lambda_i}^f(x_{i}) \quad (i= 0,1,\ldots).
\end{equation}
For our manifold case, it can be written as 
an implicit difference scheme:
\begin{equation}
x_i =  \exp_{x_{i+1}} \lambda_i \nabla f(x_{i+1}). 
\end{equation}
Several nice (convergence) properties are known 
for the sequence of (\ref{eqn:proximal}).
For example, the contraction property (\ref{eqn:contraction_D}) holds for the semigroup of (\ref{eqn:proximal}); 
see \cite[Theorem 2.2.23]{Bacak_Book2014}. 
On the other hand, solving (\ref{eqn:resolvent}) is a nontrivial task from an algorithmic point of view.
\end{Rem}    

\begin{Rem}\label{rem:descend}
In the case of $M = \RR^n$, Lemma~\ref{lem:f(x_i)}~(1) can be easily obtained from a known inequality.
For an $L$-smooth convex function $f$ in $\RR^n$, the following inequality holds (e.g., \cite[Theorem~5.8 (iii)]{Amir_Book}):
\begin{equation*}
    f(y) - f(x) \ge \langle\nabla f(x), y - x\rangle + \frac{1}{2L}\|\nabla f(x) - \nabla f(y)\|^2\quad (x, y \in \RR^n),
\end{equation*}
though we do not know a reasonable manifold version to hold.
By substituting $x = x_{i + 1}, y = x_i$, 
and using $x_i - x_{i + 1} = \nabla f(x_i)/L$ and $\|\nabla f(x_{i})\| \geq  \|\nabla f(x_{i + 1})\|$ (Lemma~\ref{lem:f(x_i)}~(2)), we have Lemma~\ref{lem:f(x_i)}~(1):
\[
    f(x_{i + 1}) \le f(x_i) - \frac{1}{2L}\left(\|\nabla f(x_{i})\|^2 + \|\nabla f(x_{i + 1})\|^2\right) \leq f(x_i) - \frac{1}{L}  \|\nabla f(x_{i + 1})\|^2.
\]
\end{Rem}

\subsection{Euclidean specialization}\label{subsec:Euclidean}
Here, we present refinements of the above results for 
the Euclidean setting $M = \RR^n$.
As far as our knowledge, 
the above convergence results 
on the gradient flow/descent
seem new even in this special case, and are further sharpened as follows.
In the Euclidean space $M = \RR^n$, 
the tangent space $T_x$ is also identified with $\RR^n$ for every $x \in M$,
where the inner product is given by $\langle u,v \rangle := u^{\top}v$.
The parallel transport $\tau_\gamma$ for any path $\gamma$ 
is the identity map. 
Let $f:\RR^n \to \RR$ be a (smooth) convex function.
We assume $L$-smoothness of $f$ when gradient descent (\ref{eqn:DGR}) is considered.
The gradient $\nabla f(x) \in \RR^n$ 
and Hessian $\nabla^2 f(x) \in \RR^{n \times n}$ are 
obtained by $(\nabla f(x))_i = (\partial/\partial x_i) f(x)$ and
$(\nabla^2 f(x))_{ij} = (\partial^2/\partial x_i \partial x_j) f(x)$, respectively.

In this setting, the strong duality (Corollary~\ref{cor:duality}) is written as
\begin{equation}\label{eqn:duality_Euclidean}
\inf_{p \in \overline{\nabla f(\RR^n)}} \|p\| = \sup_{u \in \RR^n: \|u\| \leq 1} - f^{\infty}(u),
\end{equation}
where $\overline{\nabla f(\RR^n)}$ is the closure of the gradient image
$\nabla f(\RR^n) = \{ \nabla f(x) \mid x \in \RR^n\}$.
This relation itself is deduced within Euclidean convex analysis as follows.
Let $f^*:\RR^n \to \RR \cup \{\infty\}$ be the Legendre-Fenchel conjugate of $f$:
\[
f^*(p) := \sup \{\langle p,x\rangle - f(x) \mid x \in \RR^n\} \quad (p \in \RR^n).
\]
Then, the gradient space $\overline{\nabla f(\RR^n)}$ is equal to the closure $\overline{\dom f^*}$ of the domain $\dom f^* := \{p \in \RR^n \mid f^*(p) < \infty\}$ of $f^*$. Indeed, this is because $\nabla f(\RR^n) \subseteq \dom f^* \subseteq \overline{\nabla f(\RR^n)}$, where the first inclusion follows from $p = \nabla f(x) \Leftrightarrow f^*(p) = \langle p,x\rangle-f(x)$
and the second from $f^*(p) < \infty \Leftrightarrow \inf_{x \in \RR^n} f(x) - \langle p,x \rangle > -\infty \Rightarrow \inf_{x \in \RR^n} \|\nabla f(x) -p\| = 0$.
Also, it is known in convex analysis~\cite[Theorems 13.1 and 13.3]{Rockafellar}
that $f^{\infty}$ is equal to the support function of $\dom f^*$. Summarizing, it holds 
\begin{equation}\label{eqn:closure_nablaf(RR^n)}
\overline{\nabla f(\RR^n)} = \overline{\dom f^*} = \{ p \in \RR^n \mid \langle u,p\rangle \leq f^{\infty}(u) \ (u \in \RR^{n})  \}.
\end{equation}
In particular, the gradient space $\overline{\nabla f(\RR^n)}$ is (closed) convex.
Now,  the equality in (\ref{eqn:duality_Euclidean}) is attained by the (uniquely-determined) minimum-norm point $p^*$ of $\overline{\nabla f(\RR^n)}$ and its negative direction $-p^*/\|p^*\|$; see the proof of the next theorem. By Theorems~\ref{thm:continuous_main} and \ref{thm:discrete_main}, 
both $\nabla f(x(t))$ and $\nabla f(x_i)$ 
converge to $p^*$, 
and both $x(t)$ and $x_i$ converge to $- p^*/\|p^*\|$ in cone topology.
\begin{Thm}\label{thm:mnp}
Let $p^*$ denote the minimum-norm point of $\overline{\nabla f(\RR^n)}$.
Suppose that $\kappa^* := \inf_{x \in \RR^n} \|\nabla f(x)\| >0$.
\begin{itemize}
\item[(1)] $\nabla f(x(t))$ converges to $p^*$, and $x(t)/t$ converges to $-p^*$.
\item[(2)] $\nabla f(x_i)$ converges to $p^*$, and $x_i/i$ converges to $-p^*/L$. 
\end{itemize}
\end{Thm}
\begin{proof}
It suffices to show the claims for $x(t)/t$ and $x_i/i$.
We first verify that the unique minimizer of $f^\infty$ over the unit sphere
is written as $-p^{*}/\|p^*\| =:u^*$.
Observe from the KKT-condition that $\{p \in \RR^n \mid \langle u^*,p\rangle = f^{\infty}(u^*)\}$ is a supporting hyperplane of $\overline{\nabla f(\RR^n)}$ at $p^*$.
Then, for any unit vector $v$, it holds $f^{\infty}(v) 
\geq \langle v,p^*\rangle \geq - \|p^*\| = \langle u^*,p^*\rangle  = f^{\infty}(u^*)$.
In particular, $p^*$ and $u^* = -p^*/\|p^*\|$ attain the equality in (\ref{eqn:duality_Euclidean}).

Then, by Theorem~\ref{thm:continuous_main}, we have $\lim_{t \to \infty}x(t) = - p^*/\|p^*\|$ ``in cone topology."
This implies that
\begin{equation}\label{eqn:-p*/|p*|}
\frac{-p^*}{\|p^*\|} = \lim_{t \to \infty} \frac{x(t)-x_0}{\|x(t)-x_0\|} = \lim_{t \to \infty} \frac{x(t)}{t}\frac{t}{d(x(t),x_0)} = \lim_{t \to \infty} \frac{x(t)}{t} \frac{1}{\|p^*\|},
\end{equation}
where the last equality follows from 
Proposition~\ref{prop:velocity_of_escape}
with $\|p^*\| = \lim_{t \to \infty} \|\nabla f(x(t))\| =\kappa^*$. Thus we have the latter part of (1).
The latter part of (2) is analogously shown by using  Theorem~\ref{thm:discrete_main} and Proposition~\ref{prop:velocity_of_escape_D} 
(for the sequence version of~(\ref{eqn:-p*/|p*|})).
\end{proof}
Since $-p^* = \kappa^*\xi^*$, the expected convergence in Questions~\ref{que:convergence} and \ref{que:convergence_D} hold in this case.
We end this section with other interesting aspects.
\paragraph{Hessian Riemannian gradient flow.}
Here we point out that 
the convergence of $\nabla f(x(t))$ to the minimum-norm point $p^*$ can also be explained 
via the theory of {\em Hessian Riemannian gradient flows} 
by Alvarez, Bolte, and Brahic~\cite{AlvarezBolteBrahic2004}.
Suppose for simplicity that the Hessian $\nabla^2 f(x)$ is nonsingular for every $x \in \RR^n$. 
Then, by the inverse mapping theorem applied to $x \mapsto \nabla f(x)$ (with the inverse $p \mapsto \nabla f^*(p)$),
we see that $\nabla f(\RR^n)$ is an open (convex) set. 

Consider the continuous gradient flow $x(t)$, 
and let $p(t) := \nabla f(x(t))$.
One more differentiation in (\ref{eqn:GR}) yields 
\begin{equation*}
\dot p(t) = - \nabla^2 f(x(t)) p(t).
\end{equation*}
From $\nabla^2 f(x(t)) = (\nabla^2 f^*(p(t)))^{-1}$, 
we have the following ODE obeyed by $p(t)$:
\begin{equation}\label{eqn;DE_p(t)}
\dot p(t) = - (\nabla^2 f^*(p(t)))^{-1} p(t), \quad p(0) = \nabla f(x_0).
\end{equation}
This can be interpreted as a gradient-flow ODE on a Riemannian manifold. 
Define a Riemannian metric $\langle, \rangle^{f}$ on open convex set 
$\nabla f(\RR^n)$ by
\begin{equation}
\langle u,v \rangle^f := \langle u, \nabla^2 f^{*}(p) v\rangle \quad (u,v \in T_p = \RR^n, p \in \nabla f(\RR^n)).
\end{equation}
In this metric, the gradient 
$\nabla^{f} g(p)$ of $g: \nabla f(\RR^n) \to \RR$ 
is given by $(\nabla^2 f^*(p))^{-1} \nabla g(p)$.
Then (\ref{eqn;DE_p(t)}) is viewed as
the gradient flow of the squared-norm function $p \mapsto \|p\|^2/2$:  
\begin{equation}\label{eqn:Hessian_GR}
\dot p(t) = - \nabla^f \frac{\|p(t)\|^2}{2}, \quad p(0) = \nabla f(x_0).
\end{equation}
This is a particular instance of Hessian Riemannian gradient flow in~\cite{AlvarezBolteBrahic2004}. 
Then, by \cite[Proposition 4.4]{AlvarezBolteBrahic2004}, 
the solution $p(t)$ of (\ref{eqn:Hessian_GR}) minimizes $\|p\|^2/2$ 
over $\overline{\nabla f(\RR^n)}$ in limit $t \to \infty$,
which proves $\lim_{t \to \infty} \nabla f(p(t)) = p^*$, 
the first part of Theorem~\ref{thm:mnp}~(1).

\paragraph{Mirror descent.}
On the other hand, 
the discrete version (Theorem~\ref{thm:mnp}~(2)) 
can be explained from the framework of {\em mirror descent}~\cite{NemirovskyYudin}, where 
we consult \cite[Chapter 4]{Bubeck15} for it.
Consider a general optimization problem 
\begin{equation}\label{eqn:general}
{\rm Min.}\ g(p) \quad {\rm s.t.} \quad p \in {\cal D},
\end{equation}
where $g$ is a differentiable convex function on an open convex set ${\cal D} \subseteq \RR^n$.
A {\em mirror map} $\Phi:{\cal D} \to \RR$ 
is a differentiable strictly convex function such that 
$\nabla \Phi:{\cal D} \to \RR^n$ is bijective and 
$\|\nabla \Phi(p)\| \to \infty$ if $p$ goes to the boundary of ${\cal D}$.
A basic form of mirror descent produces the sequence $p_1,p_2,\ldots$ in ${\cal D}$ 
according to the update
\begin{equation}\label{eqn:mirror}
\nabla \Phi (p_{i+1}) := \nabla \Phi(p_i) - \beta_i \nabla g(p_i),
\end{equation}
where $\beta_i > 0$ is a step size. 
It is well-known (see e.g., \cite[Section 7.4]{Vishnoi_Book}) that this update coincides with the {\em proximal gradient descent} relative to the {\em Bregman divergence} $D_{\Phi}(q,p) 
:= \Phi(q) - \Phi(p) - \langle \nabla \Phi(p),q-p \rangle$:
\begin{equation}
p_{i+1} \in \argmin_{p \in {\cal D}} \left\{ g(p_i) + \langle \nabla g(p_i), p - p_i \rangle + \frac{1}{\beta_i} D_{\Phi}(p,p_i) \right\}.
\end{equation}
Under several assumptions on $g, \Phi$, 
the solution $p_i$ (or the average solution $(1/i)\sum_{j=1}^i p_j$ or the best solution ever) is shown to converge to a minimizer of $g$;  
see e.g., \cite{LuFreundNestrov2018}, \cite[Chapter 7]{Vishnoi_Book}, \cite[Theorem 4.2]{Bubeck15}, and \cite[Section 9.2]{Amir_Book}. 

Now, consider the setting $g(p) := \|p\|^2/2$ and ${\cal D} := \nabla f(\RR^n)$.
That is, (\ref{eqn:general}) is the minimum-norm point problem on $\overline{\nabla f(\RR^n)}$.
As a mirror map, we can choose the Legendre-Fenchel conjugate $\Phi := f^*\mid_{\cal D}$. 
Then, the update (\ref{eqn:mirror}) becomes
\begin{equation}\label{eqn:mirror2}
\nabla f^* (p_{i+1}) := \nabla f^* (p_i) - \beta_i p_i.
\end{equation}
Define $x_i \in \RR^n$ by $x_i := \nabla f^* (p_{i})$. 
Since $p_i = \nabla f(x_i)$,  (\ref{eqn:mirror2}) becomes
\begin{equation}\label{eqn:mirror3}
x_{i+1} := x_i - \beta_i \nabla f(x_i).
\end{equation}
This is nothing but gradient descent, where the above Hessian Riemannian gradient flow is viewed as 
the continuous limit $\nabla^2 f^*(p(t))\dot p(t) = -p(t)$ of (\ref{eqn:mirror2}).
Then, the first part of Theorem~\ref{thm:mnp} (2) 
can be deduced from \cite[Theorem 3.1]{LuFreundNestrov2018}. 
Furthermore, an $O(1/i)$ convergence rate is obtained if $f^*(p^*) < \infty$ ($\Leftrightarrow$ $D_{f^*}(p^*,p) < \infty$). 
See \cite{Sakabe} for details.

It may be interesting to develop a manifold analogy of these observations, 
which may use the space $\nabla^{\infty}f(M) \subseteq CM^{\infty}$ in \cite{Hirai_Hadamard2022}.
Related to this issue, in Section~\ref{subsec:norm_minimization}, 
we will consider an analogous gradient flow (Kirwan's flow) 
in the complex projective space $\mathbb{P}(V)$.

\paragraph{Matrix scaling and geometric programming.}
The {\em matrix scaling problem}~\cite{Sinkhorn1964} is: For a given nonnegative matrix 
$A  = (a_{ij}) \in \RR_+^{n \times n}$, find positive diagonal matrices (scaling matrices) $X,Y$ 
such that $XAY$ approximates a doubly stochastic matrix, i.e.,
$\|(XAY){\bf 1}-{\bf 1}\| \approx 0$ and $\|(XAY)^{\top}{\bf 1} - {\bf 1}\| \approx 0$.
Define a convex function $f_{A}:\RR^n \times \RR^n \to \RR$ by
\begin{equation}\label{eqn:matrixscaling}
f_A(x,y) := \log \sum_{i,j} a_{ij} e^{x_i + y_j} - {\bf 1}^{\top}x/n  - {\bf 1}^{\top}y/n  \quad (x \in \RR^n,y \in \RR^n).
\end{equation}
From $\nabla f_A(x,y) = 
(  XAY{\bf 1} -{\bf 1}, (XAY)^{\top}{\bf 1} - {\bf 1} )/n$ 
for $(X,Y) := (e^{\diag x}, e^{\diag y}) \sqrt{n/\sum_{i,j} a_{ij} e^{x_i+y_j}}$, the required scaling matrices $X,Y$ are obtained from
$(x,y)$ having small gradient norm $\|\nabla f_A(x,y)\|$.
Particularly, such a point $(x,y)$ 
is obtained by minimizing $f_A$.

This matrix-scaling optimization falls into a more general class of
convex optimization, called {\em geometric programming}, to which our results are applicable.
A geometric program asks to minimize
a function $f: \RR^n \to \RR$ of the following form:
\begin{equation}\label{eqn:gp}
f(x) = \log \sum_{\ell =1}^N a_\ell e^{\omega_\ell^{\top}x} \quad (x \in \RR^n),
\end{equation}
where $a_\ell > 0$ and $\omega_\ell \in \RR^n$ for $\ell =1,2,\ldots,N$. 
It is well-known (see e.g.,\cite{BLNW2020interiorpoint}) that
\begin{itemize}
\item $f$ is $L$-smooth convex with $L := \max_{\ell} \|\omega_\ell\|^2$, and
\item $\overline{\nabla f(\RR^n)} = \Conv \{\omega_\ell\}_{\ell \in [N]}$.
\end{itemize}
Therefore, with $L = 2$, by gradient descent (\ref{eqn:DGR}) 
applied to (\ref{eqn:matrixscaling}), the gradient sequence 
$\nabla f_A(x_i)$ converges to the minimum-norm point $p^*$ of $\Conv \{e_i+e_j \mid i,j: a_{ij} > 0\}$.

We will show  in Section~\ref{subsec:operatorscaling} for the general setting of operator scaling that
the point $p^*$ and the limit of $XAY$
are characterized by a canonical block-triangular form of $A$,  
known as (an extended version of) the {\em DM-decomposition}~\cite{DulmageMendelsohn1958}; 
see also \cite[Section 2.2.3]{MurotaMatrixMatroid}. 
A similar convergence property was earlier shown by Hayashi, Hirai, and Sakabe~\cite{HayashiHiraiSakabe} 
for the {\em Sinkhorn algorithm}~\cite{Sinkhorn1964}, the standard alternating minimization algorithm 
for (\ref{eqn:matrixscaling}), 
in which the gradient $\nabla f_A(x,y)$
and the scaled matrix $XAY$ oscillate between two limit points described by the DM-decomposition.

\section{Application}\label{sec:applications}


\subsection{Norm-minimization in reductive group action}\label{subsec:norm_minimization}
We consider the formulation of {\em noncommutative optimization} in \cite{BFGOWW_FOCS2019}; see also \cite{HiraiNieuwboerWalter2023FOCS}.
Let $G \subseteq GL_n$ be 
a connected reductive algebraic group over $\CC$, where 
we assume that it is self-adjoint $G = G^{\dagger}$ (via conjugation~\cite[Theorem 3.13]{Wallach_book}).
Its Lie algebra $\mathfrak{g}$ is the complexification of the Lie algebra 
$\mathfrak{k}$
of a maximal compact subgroup $K = G \cap U_n$ 
as $\mathfrak{g} = \mathfrak{k} + i \mathfrak{k}$, where $i\mathfrak{k} \subseteq \mathfrak{p_n}$.
The inner product $\langle, \rangle$ on $\mathfrak{g}$ is defined by 
$\langle X,Y \rangle := Re \trace XY^{\dagger}$. 
Let $V$ be a finite-dimensional vector space over $\CC$.
Let $\pi: G \to GL(V)$ be a rational representation, 
where $\Pi$ denotes its Lie algebra representation: $\Pi(X) := (d/dt) \pi(e^{tX}) \mid_{t = 0}$.
Consider a $K$-invariant Hermitian inner product $(, )$ 
and the associated norm $\|\cdot \| = \sqrt{( \cdot, \cdot )}$ on $V$.
The norm-minimization problem over the orbit $\pi(G)v$ of $v \in V \setminus \{0\}$ is given by
\begin{equation}
\mbox{inf.}  \quad \|\pi(g) v\| \quad \mbox{s.t.}\quad g \in G.
\end{equation}
It turned out (e.g., \cite{BFGOWW_FOCS2019}) that  
this class of optimization problems has numerous, sometimes unexpected, applications 
and connections in various fields of mathematical sciences.
The most fundamental problem is to ask whether the infimum is zero, i.e., whether the origin  $0$ is in the orbit closure $\overline{\pi(G)v}$. 
This is the semistability problem in 
geometric invariant theory (GIT).
The representation $\pi$ gives rise to 
a Hamiltonian action $(g,[v]) \mapsto [\pi(g)v]$ on the complex projective space $\mathbb{P}(V)$.
The corresponding (modified\footnote{ 
The formal definition of the moment map 
is given by $[v] \mapsto - i \mu([v]) \in \mathfrak{k}$~\cite[Lemma 8.2]{GRS_Book}.
}) {\em moment map} $\mu:V \to i\mathfrak{k}$ is given by
\begin{equation}
 \langle \mu(v), H \rangle  := \frac{( v, \Pi(H)v )}{( v,v )} \quad (v \in V, H \in i\mathfrak{k}),
\end{equation} 
where $\mu$ may be regarded as $\mathbb{P}(V) \to i\mathfrak{k}$.
The following theorem is fundamental:
\begin{Thm}[{Kempf-Ness theorem, Hilbert-Mumford criterion; see \cite[Theorem 8.5 (i), Theorem 12.4]{GRS_Book}}]\label{thm:KNHM}
For $v \in V \setminus \{0\}$,  
the following conditions are equivalent:
\begin{itemize}
\item[(i)] $\inf_{g \in G} \|\pi(g)v\| = 0$.
\item[(ii)] $\inf_{g \in G} \|\mu (\pi(g)v)\| > 0$.
\item[(iii)] There is a 1-parameter subgroup $t \mapsto e(t)$ of $G$ 
such that $\lim_{t \to \infty} \pi(e(t))v = 0$.
\end{itemize}
\end{Thm}
The orbit $\pi(G)v$ in this situation is called {\em unstable}. Otherwise, it is called {\em semistable}.
Accordingly, we call the 1-parameter subgroup $e(t)$ in (iii) 
a {\em destabilizing 1-PSG}.

The unstability corresponds to the lower-unboundedness of 
the {\em Kempf-Ness function} $F_v$ on the group $G$ defined by
\begin{equation}
F_v(g) := \frac{1}{2}\log \|\pi(g)v\|^2 \quad (g \in G). 
\end{equation}
Since $\|\cdot\|$ is $K$-invariant, 
the Kempf-Ness function is viewed as a function on the symmetric space $K \backslash G$.
By $\|\pi(g)v\|^2 = (\pi(g^{\dagger}g)v,v)$ and $K \backslash G \simeq P_n \cap G$ by $Kg \mapsto g^{\dagger}g$, 
we may consider the following version of the Kempf-Ness function $f_v$ on $P_n \cap G$:
\begin{equation}\label{eqn:F_v}
f_v(x) := \log ( \pi(x)v,v) \quad (x \in P_n \cap G).
\end{equation}
It is clear that $f_v(g^{\dagger}g) = 2 F_v(g)$.
Then, $f_v$ is an $L$-smooth convex function such that
the transported gradient of $f_v$ provides the moment map $\mu$:
\begin{Lem}[\cite{BFGOWW_FOCS2019}] \label{lem:grad=mu}
\begin{itemize}
\item[(1)] $f_v$ is $N_{\pi}^2$-smooth convex, where $N_\pi$ is the maximum of the norm of a weight for $\pi$.
\item[(2)
] $\tau_{x \to I} \nabla f_v(x) = \mu(\pi(x^{1/2})v)$.
\end{itemize}
\end{Lem}
The second property (2) is implicit in \cite{BFGOWW_FOCS2019} and follows from $\tau_{x \to I}H = x^{-1/2}Hx^{-1/2}$ and $
\langle \nabla f_v(x),H\rangle_x = (d/dt) f_v(x^{1/2}e^{tx^{-1/2}Hx^{-1/2}}x^{1/2}) \mid_{t = 0}= \langle \mu(\pi(x^{1/2})v), x^{-1/2}Hx^{-1/2} \rangle_I
$.
In particular, for the Kempf-Ness function $f_v$, 
the unboundedness is equivalent to the positivity 
of the minimum gradient-norm. 
Applying Corollary~\ref{cor:duality}, we have:
\begin{Thm}\label{thm:inf_sup_KempfNess}
    $\displaystyle \inf_{g \in G} \|\mu(\pi(g)v)\| = \sup_{\xi \in B_I} - f_v^{\infty}(\xi)$. If $f_v^{\infty}(\xi) < 0$, 
    then $t \mapsto e^{t \xi}$ is a destabilizing 1-PSG.
\end{Thm}
\begin{proof}
$\inf_{g \in G} \|\mu(\pi(g)v)\| = \inf_{x \in P_n \cap G} \|\mu(\pi(x^{1/2})v)\|$ follows from
$\mu(\pi(ug)v) = u \mu(\pi(g)v)u^{\dagger}$ for $u \in K$,
the polar decomposition $g = u x$ for $u \in K$, $x \in P_n \cap G$, 
and $x \in P_n \cap G \Rightarrow x^{a} \in P_n \cap G$ (since $G$ is algebraic).
The latter part can be seen from the definitions of the Kempf-Ness function (\ref{eqn:F_v}) 
and the recession function (\ref{eqn:recession_function}).
\end{proof}
As seen below, this is a part of the theory of moment-weight inequality~\cite{GRS_Book}, in which
the recession function $f_v^{\infty}$ is essentially 
Mumford's numerical invariant, called the {\em $\mu$-weight}; see Lemma~\ref{lem:mu-weight} below.

Consider applying gradient descent to $f_v$: 
\begin{equation}
x_{k+1} = \exp_{x_k} \left(- \frac{1}{L} \nabla f_v(x_k)\right),\quad x_0 = I,\label{eqn:DGR_F}
\end{equation}
where $L := N_\pi^2$. In this setting, 
updating group elements $g_k$ in $G$ may be more suitable:
\begin{equation}\label{eqn:DGR_G}
g_{k+1} = e^{-\frac{1}{2L}\mu(\pi(g_k)v)}g_k, \quad g_0=I. 
\end{equation}
This is the {\em first order algorithm} in B\"{u}rgisser et al.~\cite{BFGOWW_FOCS2019}. 
Each of the two updates (\ref{eqn:DGR_F}) and (\ref{eqn:DGR_G}) has its own advantage. 
Their relation 
is given by
\begin{Lem}\label{lem:x_k-g_k}
$x_k = g_k^{\dagger}g_k.$ 
\end{Lem}
\begin{proof}
If $g_+ = e^{-\frac{1}{2L}\mu(\pi(g)v)}g$ and $g = u x^{1/2}$ for $u \in K$, $x\in P_n \cap G$, 
then it holds
$g_{+}^{\dagger} g_{+} = g^{\dagger} e^{- \frac{1}{L} \mu(\pi(g)v)} g = x^{1/2} u^{\dagger} e^{- \frac{1}{L} \mu(\pi(u x^{1/2})v)} u x^{1/2}
=  x^{1/2} e^{- \frac{1}{L} \mu(\pi(x^{1/2})v)} x^{1/2} = \exp_{x} - \frac{1}{L} \nabla f_v(x)$,
where the third inequality follows from 
$\mu(\pi(u)v') = u \mu(v')u^{\dagger}$ and the fourth from Lemma~\ref{lem:grad=mu}~(2).
\end{proof}
For the semistable case,
\cite{BFGOWW_FOCS2019} showed its iteration complexity to compute $\inf_{g \in G}\|\pi(g)v\|$ and to find $g \in G$ with $\| \mu(\pi(g)v)\|\approx 0$. 
For the unstable case, 
our result (Theorem~\ref{thm:discrete_main}) implies
that gradient descent (\ref{eqn:DGR_F}) constructs 
a destabilizing $1$-PSG in the limit, which
is {\em maximally destabilizing} in the sense that it is obtained from the unique minimizer of $f^{\infty}_v$ 
over~$S_I (P_n \cap G)$ (recall that $S_I$ denotes the unit sphere in $T_I$). This special 1-PSG is the same as the one shown by Kempf~\cite{Kempf1978}.
\begin{Thm}\label{thm:convergence_to_1-PSG} 
Suppose that $\inf_{g \in G} \|\pi(g)v\| = 0$.
Let $x_k$ be the sequence of (\ref{eqn:DGR_F}), and
let $u_k$ be the sequence defined by $x_k = e^{d(x_k,I)u_k}$. 
Then $u_k$ converges to the unique minimizer $\xi^*$ of $f_v^{\infty}$ over $S_I$, where 
$t \mapsto e^{t\xi^*}$ is a maximally destabilizing 1-PSG.
\end{Thm}
Unfortunately, since $f^{\infty}_v$ is 
not necessarily (upper semi)continuous, 
this theorem 
does not imply the algorithmic statement: $t \mapsto e^{t u_k}$ 
is a destabilizing $1$-PSG for some large $k$.
Therefore,
we need a certain rounding idea 
to obtain a destabilizing $1$-PSG from $u_k$. 
We see in the next Section~\ref{subsec:operatorscaling} that 
such a rounding is possible 
for  the left-right action.

We also consider convergence of the moment-map sequence $\mu(\pi(g_k)v)$.
Let ${\cal C}_{\pi} \subseteq i\mathfrak{k} = T_I(P_n \cap G)$ denote 
a positive Weyl chamber: It is a convex cone with the property that 
for any $H \in i \mathfrak{k}$ there is a unique 
point in ${\cal C}_\pi$, denoted by $\spec H$, satisfying $\spec H = k H k^{\dagger}$ 
for some $k \in K$.
The {\em moment polytope} $\varDelta_{v} \subseteq {\cal C}_\pi$ is defined as 
the closure of the image of $g \mapsto \spec \mu(\pi(g)v)$:
\[
\varDelta_{v} := \overline{\{\spec \mu(\pi(g)v) \mid g \in G\}}.
\]
The convexity theorem by Guillemin and Sternberg~\cite{GuilleminSternberg1982,GuilleminSternberg1984} and Kirwan~\cite{Kirwan1984} 
says that it is a convex polytope.
\begin{Thm}[{Convexity theorem~\cite{GuilleminSternberg1982,GuilleminSternberg1984,Kirwan1984}}]
$\varDelta_v$ is a convex polytope.
\end{Thm}
By Lemma~\ref{lem:grad=mu}~(2), the polar decomposition $g = u x^{1/2}$ for $g \in G$, 
$u \in K$, $x \in P_n \cap G$, and $\mu(\pi(ux^{1/2})v) = u\mu(\pi(x^{1/2})v)u^{\dagger}$, 
it holds
\begin{equation}\label{eqn:|grad|=|mu|=|p|}
\inf_{x \in P_n \cap G} \|\nabla f_v(x)\|  =  \inf_{x \in P_n \cap G} \|\mu(\pi(x^{1/2})v)\| 
= \inf_{g \in G} \|\mu(\pi(g)v)\|= \inf_{g \in G} \|\spec \mu(\pi(g)v)\| = \inf_{p \in \varDelta_v} \|p\| 
\end{equation}
By Theorem~\ref{thm:discrete_main}, 
we have the convergence of $\spec \mu(\pi(g_k) v) (= \spec \mu(\pi(x_k^{1/2}) v))$ along the gradient-descent trajectory, which is an analogue of Theorem~\ref{thm:mnp}~(2).
\begin{Thm}\label{thm:convergene_of_momentmap} 
Let $p^*$ be the minimum-norm point of $\varDelta_v$, and let $H_k$ be the sequence defined by $x_k = e^{k H_k/L}$. 
Suppose that $\inf_{g \in G} \|\pi(g)v\| = 0$.
Then, both $\spec \mu(\pi(g_k) v)$ and $\spec (- H_k)$ converge to $p^*$ for $k \to \infty$.
\end{Thm}
\begin{proof}
It suffices to show the claim for $H_k$. By Proposition~\ref{prop:velocity_of_escape_D} (2), Proposition~\ref{prop:accumulation_of_transported_gradient_D}, and Lemma~\ref{lem:grad=mu}~(2), 
it holds 
\[
\liminf_{k \to \infty}\| \mu(\pi(x^{1/2}_k)v) + H_k \| = 0.
\]
Since $\spec \mu(\pi(x_k^{1/2})v)$ converges to $p^*$ and $H_k$ converges (to $\kappa^*\xi^*$), 
it must hold that $\spec (-H_k)$ converges to $p^*$.
\end{proof}
Question~\ref{que:convergence_D}, if it is true, 
would imply the stronger convergence $\lim_{k \to \infty} \mu(\pi(x^{1/2}_k)v) = - \lim_{k\to \infty}H_k$.


\paragraph{Moment-weight inequality and gradient flow of moment-map squared.}
Clearly, via Theorem~\ref{thm:continuous_main}, 
the above results (Theorems~\ref{thm:convergence_to_1-PSG} and \ref{thm:convergene_of_momentmap}) hold 
for the gradient flow: 
\begin{equation}\label{eqn:GR_F}
\dot x(t) = - \nabla f_v(x(t)), \quad x(0) = I.
\end{equation}
Our consideration of this case falls into the theory 
of {\em moment-weight inequality} by Georgoulas, Robbin, and Salamon~\cite{GRS_Book}, 
which builds upon the earlier work by Kempf, Kirwan, Mumford, and Ness in GIT, 
and the recent work by Chen and Sun~\cite{ChenSun2014} in $K$-stability.
Here, we briefly summarize the relation 
by deducing an important part of the theory
from our results in Section~\ref{subsec:continuous_GR}.
We use notation $g \cdot [v] := [\pi(g)v]$ for the action on $\mathbb{P}(V)$. 
According to \cite[Chapter 3]{GRS_Book},
consider the gradient flow ({\em Kirwan's flow}) of the squared-norm of the moment map 
on $\mathbb{P}(V)$: 
\begin{equation}\label{eqn:KirwanGR}
\dot \zeta(t) = - \nabla \frac{\|\mu(\zeta(t))\|^2}{2} , \quad \zeta(0) = [v].
\end{equation}
This is the gradient flow 
of a real analytic function $\zeta \mapsto \|\mu(\zeta)\|^2/2$ on 
a compact Riemannian manifold $\mathbb{P}(V)$ (with respect to the Fubini-Study metric).
By the standard argument of the {\L}ojasiewicz gradient inequality, 
the limit of $\zeta(t)$ exists.
\begin{Thm}[{Convergence Theorem~\cite[Theorem 3.3]{GRS_Book}}]\label{thm:limit}
    The limit $\zeta_{\infty} := \lim_{t \to \infty} \zeta(t)$ exists.
\end{Thm}

Further, the limit $\zeta_{\infty}$ attains the infimum of the moment-map norm 
over the orbit $G \cdot [v]$ in $\mathbb{P}(V)$.
\begin{Thm}[{Moment-limit theorem \cite[Theorem 6.4]{GRS_Book}}]\label{thm:moment_limit}
Let $\zeta(t)$ be the solution of (\ref{eqn:KirwanGR}), 
and let $\zeta_{\infty} := \lim_{t \to \infty} \zeta(t)$.
Then it holds 
\begin{equation}\label{eqn:mu(zeta_infty)}
\|\mu(\zeta_{\infty})\| = \inf_{g \in G} \|\mu(g \cdot [v])\|.
\end{equation}
\end{Thm}
The equality (\ref{eqn:mu(zeta_infty)}) can be understood from Theorem~\ref{thm:continuous_main} as follows.
Regard $G$ as a Riemannian manifold by the right-invariant Riemannian metric 
$\langle X,Y\rangle_g := Re \trace Xg^{-1} (Yg^{-1})^{\dagger}$ for $X,Y \in T_g, g \in G$, consider the gradient flow of $F_v$ on $G$:
\begin{equation}\label{eqn:GR_G}
\dot g(t) = -\nabla F_v(g(t)), \quad g(0) = I.
\end{equation}
Then, the solution $\zeta(t)$ is obtained from the action of $g(t)$ as follows:
\begin{Thm}[{\cite[Theorem 4.1 (ii)]{GRS_Book}}]\label{prop:relation}
The solution $\zeta(t)$ of (\ref{eqn:KirwanGR}) 
is represented as $\zeta(t) = g(t) \cdot [v]$ 
for the solution $g(t)$ of (\ref{eqn:GR_G}). 
\end{Thm}
\begin{proof}[Proof sketch]
Define $\varphi:G \to \mathbb{P}(V)$ by $g \mapsto g\cdot [v]$. 
Then, by adapting \cite[(4.3)]{GRS_Book} with our notation, it holds 
$d \varphi_{g} \nabla F_v(g) = \nabla \frac{\|\mu(g \cdot [v])\|^2}{2}$.
Thus $(d/dt) (g(t) \cdot [v]) = (d/dt) \varphi(g(t)) = d\varphi_{g(t)} \dot g(t) =  
- d\varphi_{g(t)} \nabla F_v(g(t)) = - \nabla \frac{\|\mu(g(t) \cdot [v])\|^2}{2}$, 
implying that 
$g(t) \cdot [v]$ is the solution $\zeta(t)$ of (\ref{eqn:KirwanGR}).
\end{proof}
We can see that $\nabla F_v(g) = \mu(\pi(g)v)g$ and 
(\ref{eqn:DGR_G}) is the discretization (gradient descent) of (\ref{eqn:GR_G}). 
Analogously to Lemma~\ref{lem:x_k-g_k},
the relation between $x(t)$ and $g(t)$ is given by
\begin{Lem}\label{lem:x(t)-g(t)}
$x(2t) = g(t)^{\dagger}g(t).$
\end{Lem}
\begin{proof}
For $H \in T_{g^\dagger g}(P_n \cap G)$, 
it holds $\langle \nabla f_v(g^{\dagger}g),H\rangle_{g^{\dagger}g} = \frac{d}{dt}\mid_{t = 0} f_v(g^{\dagger}e^{t g^{-\dagger} H g^{-1}}g) =  \frac{d}{dt}\mid_{t =0} 2 F_v(e^{t g^{-\dagger} H g^{-1}/2}g) = \langle \nabla F_v(g), g^{-\dagger}H \rangle_{g}  
= \langle  g^{\dagger} \nabla F_v(g) + \nabla F_v(g)^{\dagger}g,H/2\rangle_{g^{\dagger}g}$.
Hence, it holds that $2 \nabla f_v(g^{\dagger}g) =  g^{\dagger} \nabla F_v(g) + \nabla F_v(g)^{\dagger}g$, and 
\begin{equation*}
\frac{d}{ds}(g(s)^{\dagger}g(s)) = \dot g(s)^{\dagger} g(s) + g(s)^{\dagger} \dot g(s) = 
 - \nabla F_v(g(s))^{\dagger}g(s) - g(s)^{\dagger} \nabla F_v(g(s)) = - 2 \nabla f_v(g(s)^{\dagger}g(s)).
\end{equation*}
Thus $x(t) := g(t/2)^{\dagger}g(t/2)$ satisfies (\ref{eqn:GR_F}).
\end{proof}
Therefore, 
the moment-limit theorem (Theorem~\ref{thm:moment_limit}) follows 
from $\|\mu(\zeta_{\infty})\| = \lim_{t \to \infty} \|\mu (\pi(g(t))v)\| 
= \lim_{t \to \infty} \|\mu (\pi(x(t)^{1/2})v)\| = \inf_{x \in P_n \cap G} \|\nabla f_v(x(t))\| 
= \inf_{g \in G} \|\mu(g\cdot [v])\|$.
Accordingly, an analogue of Theorem~\ref{thm:mnp} (1) (or the continuous version of Theorem~\ref{thm:convergene_of_momentmap}) is the following.
\begin{Thm}\label{thm:limit_of_momentmap} 
Let $p^*$ be the minimum-norm point of $\varDelta_v$, and let $H(t)$ be the function defined by $x(t) = e^{t H(t)}$. 
Suppose that $\inf_{g \in G} \|\pi(g)v\| = 0$.
Then, both $\spec \mu(\pi(g(t)) v)$ and $\spec (-H(t))$ converge to $p^*$ for $t \to \infty$.
\end{Thm}

\begin{proof}
It suffices to show the claim for $H(t)$. 
By Proposition~\ref{prop:velocity_of_escape}~(2), Proposition~\ref{prop:accumulation_of_transported_gradient}, and Lemma~\ref{lem:grad=mu}~(2), it holds 
\begin{equation}
\liminf_{t \to \infty} \| \mu( \pi(x(t)^{1/2})v) + H(t) \|  = 0.
\end{equation}
The rest is the same as in the proof of Theorem~\ref{thm:convergene_of_momentmap}.
\end{proof}
Contrary to $g(t) \cdot [v]$, 
we do not know whether $x(t)^{1/2}\cdot [v]$ converges\footnote{In the earlier versions of this paper, the convergence of $x(t)^{1/2}\cdot [v]$ was stated but the proof was false.}. 
At least, if Question~\ref{que:convergence} is affirmative, 
then $\mu(\pi(x(t)^{1/2})v)$ converges.
On the other hand, 
$\mu(g(t) \cdot [v])$ converges to $\mu(\zeta_{\infty})$, $- H(t)$ converges to $- H_{\infty} (= - \kappa^*\xi^*)$, and they have the same spectrum $p^*$.
Therefore, there is $u_{\infty} \in K$ such that $u_{\infty} \mu(\zeta_{\infty}) u_{\infty}^\dagger = - H_{\infty}$. This fact is a part of the {\em generalized Kempf existence theorem}~\cite[Theorem 10.4, (10.9)]{GRS_Book}.  
In particular, Theorem~\ref{thm:convergene_of_momentmap}
can be viewed as a discrete version of the moment-limit theorem, 
though we do not know whether $\zeta_k := g_k \cdot [v]$ converges.

We next explain the moment-weight inequality.
The (restriction of) {\em $\mu$-weight} $w_\mu : \mathbb{P}(V) \times i\mathfrak{k} \to \RR  \cup \{\infty\}$ is defined by
\begin{equation}
w_{\mu}([v],H) :=\lim_{t \to \infty} \trace \mu(\pi(e^{tH})v) H \quad ([v] \in \mathbb{P}(V), H \in  i\mathfrak{k}= T_I(P_n \cap G)),
\end{equation}
where the existence of the limit is seen in the proof of the next lemma.
The $\mu$-weight is nothing but the recession function of $f_v^{\infty}$.
\begin{Lem}[{See \cite[Lemma 5.2]{GRS_Book}}]\label{lem:mu-weight}
$w_{\mu}([v],H) =  f^{\infty}_v(H)$.
\end{Lem}
\begin{proof} 
By recalling (\ref{eqn:recession_function}), 
it holds $f^{\infty}_v(H) = \lim_{t \to \infty} (d/dt) f_v(e^{tH}) 
= \lim_{t \to \infty} \trace \mu (\pi(e^{tH})v) H =  w_{\mu}([v],H)$, 
where the second equality follows from Lemma~\ref{lem:grad=mu}~(2).
\end{proof}
We now state the main part of the theory of moment-weight inequality (for linear actions).  
\begin{Thm}[{Moment-weight inequality \cite[Theorems 6.7, 10.1, 10.2, 10.4]{GRS_Book}}]\label{thm:moment_weight}
It holds
\begin{equation}\label{eqn:moment_weight}
\inf_{g \in G} \|\mu(g \cdot [v])\| \geq \sup_{H \in i\mathfrak{k} \setminus \{0\}}  \frac{- w_{\mu}([v],H)}{\|H\|}.
\end{equation}
Suppose that $\kappa^* := \inf_{g \in G} \|\mu(g \cdot [v])\| > 0$. 
Then the equality in (\ref{eqn:moment_weight}) holds, and the supremum is attained by unique $H^* \in i\mathfrak{k}$ with $\|H^*\| =1$, obtained as follows:
Let $H(t)$ be defined by
$
x(t) = e^{tH(t)}
$
for solution $x(t)$ of (\ref{eqn:GR_F}). Then the limit $H_{\infty} := \lim_{t \to \infty}H(t)$ exists, $\|H_{\infty}\|  = \kappa^*$, and $H^* = H_{\infty}/\|H_{\infty}\|$.
\end{Thm}
From our convex-optimization perspective, 
the moment-weight inequality (\ref{eqn:moment_weight}) is explained by the weak duality (Lemma~\ref{lem:weakduality}).
The equality case is explained by the strong duality (Theorem~\ref{thm:continuous_main}), the gradient-flow construction of the unique minimizer of $f^{\infty}_v$, and 
the formula of the velocity of escape (Proposition~\ref{prop:velocity_of_escape}).

We finally state one well-known important uniqueness property of minimizers of 
the moment-map norm over $\overline{G \cdot [v]}$, 
\begin{Thm}[{Second Ness uniqueness theorem~\cite[Theorem 6.5]{GRS_Book}}]\label{thm:Ness_second}
For $\zeta,\zeta' \in \overline{G\cdot[v]}$, if 
$\|\mu(\zeta)\| = \|\mu(\zeta')\| = \inf_{g \in G} \|\mu(g \cdot [v])\|$, 
then $\zeta' \in K \cdot \zeta$.
\end{Thm}
In the next subsection,  
we characterize such minimizers for the left-right action.

\subsection{Operator scaling and its gradient-flow limit}\label{subsec:operatorscaling}

Let $A = (A_1,A_2,\ldots,A_N) \in \CC^{N(n \times m)}$ be an $N$-tuple of $n \times m$ matrices over $\CC$.
Let  $p \in \RR^n_+$, $q \in \RR^m_+$ be nonnegative vectors
with the same sum $\sum_{i} p_i = \sum_j q_j$, where 
$p, q$ are arranged as
\begin{equation}\label{eqn:decreasing_increasing}
    p_1 \geq p_2 \geq \cdots \geq p_n, \quad  q_1 \leq q_2 \leq  \cdots \leq q_m.
\end{equation}
The {\em operator scaling} problem, originally introduced by Gurvits~\cite{Gurvits2004} 
for $p =q = {\bf 1}$ and extended by Franks~\cite{Franks2018} for general $p,q$,  is to ask: 
For a given accuracy $\epsilon \geq 0$, find $g \in GL_n, h \in GL_m$ 
such that
\begin{equation}\label{eqn:(p,q)-scalable}
\left\| \sum_{\ell=1}^N g A_\ell h^{\dagger}h A_\ell^{\dagger} g^{\dagger} - \diag p \right\|^2 + 
\left\| \sum_{\ell=1}^N h A_\ell^{\dagger} g^{\dagger}g A_\ell h^{\dagger} - \diag q \right\|^2 \leq \epsilon^2,
\end{equation}
where the norm is the Frobenius norm.
A matrix tuple $A$ is said to be (approximately) {\em $(p,q)$-scalable} 
if for every positive $\epsilon > 0$ there are $g \in GL_n,h \in GL_m$ 
satisfying (\ref{eqn:(p,q)-scalable}).
If some $g,h$ satisfy (\ref{eqn:(p,q)-scalable}) for $\epsilon = 0$, then $A$ 
is called {\em exactly $(p,q)$-scalable}, and $gAh^{\dagger}$ 
is called a {\em $(p,q)$-scaling} of $A$.
The operator scaling is a quantum generalization of the matrix scaling, 
and turned out to have rich applications; see \cite{Franks2018,GGOW_BrascampLieb,GGOW,GargOliveira2018}.
For simplicity, we assume that the left and right common kernels of $A$ are both trivial:
$\bigcap_{\ell}\ker A_\ell = \{0\}$ and $\bigcap_{\ell} \ker A^{\dagger}_\ell = \{0\}$. 

In view of the previous section, 
the operator scaling is  interpreted  
as the moment polytope membership of  
the {\em left-right action} $\pi: SL_n \times SL_m \to GL(\CC^{N(n \times m)})$ defined by
\begin{equation}
\pi(g,h)(B) := gBh^{\dagger} = (gB_1h^{\dagger}, gB_2h^\dagger,\ldots, gB_N h^{\dagger}),
\end{equation}
where $B = (B_1,B_2,\ldots,B_N) \in \CC^{N(n \times m)}$.
A maximal compact subgroup $K$ of $SL_n \times SL_m$ is given by $SU_n \times SU_m$, 
and a $K$-invariant Hermitian product $\langle,\rangle$ on $V= \CC^{N(n \times m)}$ 
is given by $\langle B,C\rangle := \sum_{\ell=1}^{N} \trace B_\ell C_\ell^{\dagger}$.  
From $\Pi(X,Y)(B) = XB+BY^{\dagger}$, we see that 
the moment map $\mu: \CC^{N(n \times m)} \to \mathfrak{p}^1_n \times \mathfrak{p}^1_m$ is given by
\begin{equation}\label{eqn:mu(B)}
\mu(B)= (\mu_1(B),\mu_2(B)) = \frac{1}{\|B\|^2 } \left(\sum_{\ell=1}^N B_\ell B_\ell^{\dagger}, \sum_{\ell=1}^N B_\ell^{\dagger} B_\ell\right) -\left(\frac{1}{n} I, \frac{1}{m}I \right). 
\end{equation}
A positive Weyl chamber is taken as 
the set of diagonal matrices $(\diag p, \diag q)$ with $p,q$ 
satisfying (\ref{eqn:decreasing_increasing}) and ${\bf 1}^{\top}p={\bf 1}^{\top}q=0$.
We regard it as a subset of $\RR^n \times \RR^m$.
Then the moment polytope $\varDelta_A$ consists of 
vectors of eigenvalues of $\mu(B)$ over $B \in \overline{SL_n \cdot A \cdot SL_m}$ 
(the closure  of the $SL_n \times SL_m$-orbit of $A$).
Comparing (\ref{eqn:mu(B)}) with (\ref{eqn:(p,q)-scalable}), we have:
\begin{Lem}\label{lem:(p,q)-scalable}
$A$ is $(p, q)$-scalable if and only 
if $(p/c - {\bf 1}/n, q/c - {\bf 1}/m)$ belongs to $\varDelta_A$, where $c := \sum_i p_i = \sum_j q_j$.
\end{Lem}
We consider 
the operator scaling problem for 
the most basic case: 
$(p,q) =  ({\bf 1}/n, {\bf 1}/m)$.
Then, it holds
\begin{equation*}
\mbox{$A$ is $({\bf 1}/n, {\bf 1}/m)$-scalable} \quad  
\Leftrightarrow  \quad  (0,0) \in \varDelta_A \quad \Leftrightarrow \quad \inf_{g,h} \| \mu(gAh^{\dagger})\| = 0.
\end{equation*}
Accordingly, 
the Kempf-Ness theorem (Theorem~\ref{thm:KNHM}) 
links with the $({\bf 1}/n, {\bf 1}/m)$-scaling problem, and is sharpened as follows. 
Let ${\cal S}_A$ denote the family of pairs of vector subspaces $X \subseteq \CC^n$, $Y \subseteq \CC^m$
such that $u^{\top} A_\ell \bar v = 0$ 
for all $u \in X$, $v \in Y$, $\ell \in [N]$.
This is (essentially) the same as the family of {\em independent} subspace pairs in \cite{Franks2018,FranksSomaGoemans_SODA2023}.
Although ${\cal S}_A$ is an infinite set, it turns out in Lemma~\ref{lem:E_A} that a certain maximal subset ${\cal E}_A$ of ${\cal S}_A$ is a finite set. 
\begin{Thm}[Characterization of scalability~\cite{Gurvits2004}]\label{thm:Gurvits1}
    The following are equivalent:
    \begin{itemize}
    \item[(i)] $\displaystyle \inf_{g \in SL_n,h \in SL_m} \| g A h^{\dagger} \| > 0$. 
    \item[(ii)] $A$ is $({\bf 1}/n, {\bf 1}/m)$-scalable. 
    \item[(iii)] For all $(X,Y) \in {\cal S}_A$, it holds $(1/n)\dim X + (1/m)\dim Y \leq 1$.
    \end{itemize}
\end{Thm}
This theorem was originally stated for the case $n=m$, 
in which the condition (iii) is simply written as $\dim Y \leq \dim \sum_{\ell =1}^N A_kY$ for every subspace $Y$. A subspace violating this condition is called a {\em shrunk subspace} in \cite{FranksSomaGoemans_SODA2023,GGOW,IQS2017,IQS2018}. 
The above $n \neq m$ generalization is straightforward and 
is included in more general results 
for the operator scaling with marginals by Franks~\cite{Franks2018}.

A vector-space pair $(X,Y) \in {\cal S}_A$ violating (iii)
actually gives rise to a destabilizing $1$-PSG as follows:
Choose $\sigma \in SU_n$ and $\tau \in SU_m$ 
such that the first $r$ rows of $\sigma$ span $X$ and the first $s$ rows of $\tau$ span $Y$, where $(r,s) := (\dim X,\dim Y)$.
Then one can see that $t \mapsto (e^{t \diag ({\bf 1}_{[r]} - (r/n){\bf 1})} \sigma, e^{t \diag ({\bf 1}_{[s]} - (s/m){\bf 1})}\tau)$ 
is a destabilizing $1$-PSG.

Further, the strict inequality in (iii) brings exact scalability.
\begin{Thm}[Exact scalability~\cite{Gurvits2004}]\label{thm:Gurvits2}
If  $(1/n)\dim X + (1/m)\dim Y < 1$ for all $(X,Y) \in {\cal S}_A$ 
other than $(\{0\},\CC^m)$ and $(\CC^n,\{0\})$, then $A$ is exactly $({\bf 1}/n, {\bf 1}/m)$-scalable.
\end{Thm}
The exact case corresponds to the existence of $g,h$ with $\mu(gAh^{\dagger}) = 0$.
By Lemma~\ref{lem:grad=mu}~(2), this is the case where 
the Kempf-Ness function $f_A$ has an optimum ($=$ a point of zero gradient). 
Then, Theorem~\ref{thm:Gurvits2} can be deduced from 
general property (\ref{eqn:exists}) of the recession function $f^{\infty}_A$ (given explicitly in (\ref{eqn:F_A^infty}) below).
Here, the Kempf-Ness function $f_A: P_n^1 \times P_m^1 \to \RR$ is written as
\begin{equation}
f_A(x,y) := \log \trace \sum_{\ell=1}^N xA_\ell yA_\ell^{\dagger}  \quad (x \in P_n^1,y \in P_m^1).  
\end{equation}
\begin{Lem}[\cite{BFGOWW_FOCS2019}]\label{lem:2-smooth}
$f_A$ is $2$-smooth convex.  
\end{Lem}
Now Theorem~\ref{thm:inf_sup_KempfNess} (Corollary~\ref{cor:duality}, or the moment-weight inequality (Theorem~\ref{thm:moment_weight}))
sharpens (ii) $\Leftrightarrow$ (iii) of Theorem~\ref{thm:Gurvits1} 
in the following min-max (inf-sup) form:
\begin{Thm}[Duality theorem for the scalability limit of operator scaling]\label{thm:scalability_limit}
\begin{eqnarray} %
&& \inf_{g,h}  
 \left\| \left( \sum_{\ell=1}^N gA_\ell h^{\dagger}h A_\ell^{\dagger}g^{\dagger} - \frac{1}{n} I, 
 \sum_{\ell=1}^N hA_\ell^{\dagger}g^{\dagger}g A_\ell h^{\dagger} -\frac{1}{m}I \right) \right\|  \nonumber\\
&&  = \sup_{a,b, \sigma, \tau} -  \max \{ a_i + b_j \mid \exists \ell, (\sigma A_\ell \tau^{\dagger})_{ij} \neq 0\}, \label{eqn:scalability_limit}
\end{eqnarray}
where the infimum in LHS is taken over all $g \in GL_n$, $h \in GL_m$ with $\|gAh^{\dagger}\| = 1$ and
the supremum in RHS is taken over all $\sigma \in SU_n$, $\tau \in SU_m$, 
$a \in \RR^n$, $b \in \RR^m$ with $\|(a,b)\| \leq 1$ and ${\bf 1}^{\top} a = {\bf 1}^{\top} b = 0$.
\end{Thm}
Inspired by this formula, Hirai~\cite{Hirai2025}
obtained a cleaner formula by using the trace norm instead of the Frobenius norm.
\begin{proof}
It suffices to show that $-f_A^{\infty}$ is equal to the objective function of RHS in (\ref{eqn:scalability_limit}).
Here $(G,H) \in \mathfrak{p}_n^1 \times \mathfrak{p}_m^1$ is written as 
$(G,H) = (\sigma^{\dagger} \diag a \sigma,  \tau^{\dagger} \diag b \tau)$ for 
$\sigma \in SU_n, \tau \in SU_m$, $a \in \RR^n$, $b \in \RR^m$ with
${\bf 1}^{\top}a = {\bf 1}^{\top}b=0$. Then we have
\begin{eqnarray}
f_A^{\infty}(G,H) &= & \lim_{t \to \infty} \frac{1}{t} \log \trace \sum_{\ell} e^{tG}A_\ell e^{tH} A_\ell^{\dagger} = 
\lim_{t \to \infty}\frac{1}{t} \log \sum_{\ell,i,j} |(\sigma A_\ell \tau^{\dagger})_{i j}|^2 e^{t(a_i + b_j)}  \nonumber \\
&=&  \max \{ a_i + b_j \mid \exists \ell, (\sigma A_\ell \tau^{\dagger})_{ij} \neq 0\}, \label{eqn:F_A^infty}
\end{eqnarray}
where we used $\lim_{t \to \infty} \frac{1}{t} \log \sum_{k} e^{c_k + t d_k} = \max_{k} d_k$ in the last equality.
\end{proof}

In the sequel, we assume that $A$ is not $({\bf 1}/n, {\bf 1}/m)$-scalable, 
and analyze the asymptotic behavior of gradient descent for $f_A$:
\begin{equation}
(x_{k+1},y_{k+1})  = \exp_{x_k,y_k} \left( - \frac{1}{L}\nabla  f_A(x_k,y_k)\right),\quad (x_0,y_0) = (I,I), \label{eqn:DGR_F_A}
\end{equation}
where we let $L:= 2$ by Lemma~\ref{lem:2-smooth}. 
The corresponding group update (\ref{eqn:DGR_G}) in $SL_n \times SL_m$ is given by
\begin{equation}\label{eqn:DGR_F_A_G}
(g_{k+1},h_{k+1}) = \left(e^{-\frac{1}{2L} \mu_1(g_kAh_k^{\dagger})} g_k, e^{-\frac{1}{2L}\mu_2(g_kAh_k^{\dagger})} h_k \right) \quad (g_0,h_0) = (I,I).
\end{equation}
Then $(x_k,y_k) = (g_k^{\dagger}g_k, h_k^{\dagger}h_k)$ by Lemma~\ref{lem:x_k-g_k}.
We address the following problem.
\begin{Prob}\label{prob:ABC}
Characterize the following (A), (B), and (C): 
\begin{itemize}
    \item[(A)] The limit of $\spec \mu(g_kAh_k^\dagger)\quad $ ( $=$ the minimum-norm point of $\varDelta_A$).
    \item[(B)] The limit of $(x_k,y_k)$ in cone topology $\quad$ ( $=$ the unique minimizer of $f_A^{\infty}$).
    \item[(C)] The limit of $[g_k A h_k^\dagger]$ in $\mathbb{P}(\CC^{N(n \times m)})\quad $ 
                 ( $=$ the minimizer of the moment-map norm $\|\mu\|$ over 
                 $\overline{[SL_n \cdot A \cdot SL_m]}$).
\end{itemize}
\end{Prob}

We show that these are characterized by a certain simultaneous block-triangular form of $A$.
This block-triangular form is a vector-space generalization of 
the classical {\em Dulmage-Mendelsohn decomposition}~\cite{DulmageMendelsohn1958} ({\em DM-decomposition} for short) 
for a bipartite graph and its associated matrix.
We introduce our generalized DM-decomposition in a way analogous 
to \cite[Section 3]{HayashiHiraiSakabe} for the classical setting, where
the essential idea of the construction can be partly found in Ito, Iwata, and Murota~\cite{ItoIwataMurota1994}.
Iwamasa, Oki, and Soma~\cite{IOS_ICALP2025} 
pointed out that our DM-decomposition is 
a special case of the {\em Harder-Narasimhan filtration}
for generalized Kronecker quivers.

Recall the family ${\cal S}_A$ defined before Theorem~\ref{thm:Gurvits1}.
Define a map $\phi: {\cal S}_A \to \RR^2_+$ by
$$
\phi(X,Y) := (\dim X, \dim Y) \quad ((X,Y) \in {\cal S}_A).
$$
Consider the convex hull $\Conv \phi({\cal S}_A) \subseteq \RR^2_+$; 
see the left of Figure~\ref{fig:diagram}.
\begin{figure}[t]
	\begin{center}
		\includegraphics[scale=0.48]{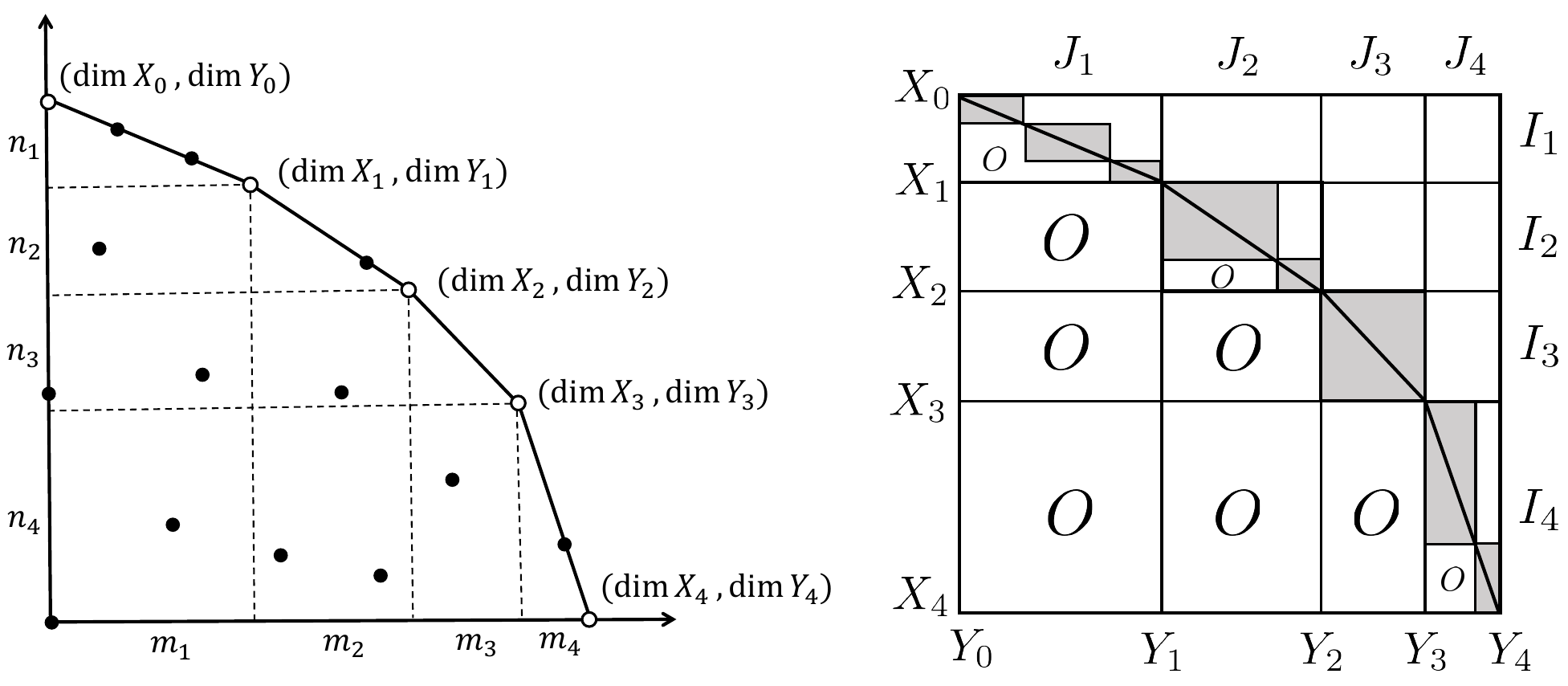}
		\caption{$\Conv \phi({\cal S}_A)$ in $(y,x)$-plane (left) 
			and a DM-decomposition of $A$ (right). The slope $n_{\alpha}/m_{\alpha}$ is increasing by the convexity of $\Conv \phi({\cal S}_A)$.}
		\label{fig:diagram}
	\end{center}
\end{figure}\noindent
Let ${\cal E}_A$ denote the subset of $(X,Y) \in {\cal S}_A$ 
such that $\phi(X,Y)$ is an extreme point of $\Conv \phi({\cal S}_A)$ not equal to $(0,0)$.
\begin{Lem}\label{lem:E_A}
	For $(X,Y),(X',Y') \in {\cal E}_A$, 
	if $\dim X \leq \dim X'$ and $\dim Y \geq \dim Y'$, 
	then $X \subseteq X'$ and $Y \supseteq Y'$.
	In particular, ${\cal E}_A$ is a finite set, and $\phi$ is injective on ${\cal E}_A$. 
\end{Lem}
\begin{proof}
We may suppose that $\phi(X,Y)$ and $\phi(X',Y')$ 
are equal or on an adjacent pair of extreme points. 
Observe $(X \cap X', Y + Y'),  (X + X', Y \cap Y') \in {\cal S}_A$.
By the dimension identity of vector spaces, it holds 
\begin{equation}\label{eqn:modular_equality}
\phi(X \cap X', Y + Y') + \phi(X + X', Y \cap Y') = \phi(X,Y)+ \phi(X',Y').
\end{equation}
We claim that $X' = X + X'$ and $Y' = Y \cap Y'$, which implies the statement. 
Otherwise, by (\ref{eqn:modular_equality}), $\phi(X \cap X', Y + Y')$ or $\phi(X + X', Y \cap Y')$ 
goes beyond $\Conv \phi({\cal S}_A)$, which contradicts $(X \cap X', Y + Y'),  (X + X', Y \cap Y') \in {\cal S}_A$.
\end{proof}

Therefore, ${\cal E}_{A} = \{(X_\alpha,Y_\alpha)\}_{\alpha=0}^\theta$ can be arranged as
\begin{eqnarray}
&& \CC^n = X_0 \supset X_1 \supset \cdots \supset X_\theta = \{0\}, \nonumber \\ 
&& \{0\} =Y_0 \subset Y_1 \subset \cdots \subset Y_\theta = \CC^m, \label{eqn:(X_alpha,Y_alpha)}
\end{eqnarray}
where $\CC^n \neq X_1$ and $Y_{\theta-1} \neq \CC^m$ follow 
from the assumption that the common left and right kernels of $A$ are trivial.
For each $\alpha \in [\theta]$, 
let ${\cal L}_A^{\alpha}$ denote the subset consisting of 
$(X,Y) \in {\cal S}_A$ such that $\phi(X,Y)$ belongs to the edge 
between $\phi(X_{\alpha-1},Y_{\alpha-1})$ and $\phi(X_\alpha,Y_\alpha)$.
As in the proof of Lemma~\ref{lem:E_A}, we have:
\begin{Lem}
    If $(X,Y),(X',Y') \in {\cal L}_A^{\alpha}$, 
    then $(X+X', Y \cap Y'), (X \cap X', Y+ Y') \in {\cal L}_A^{\alpha}$.
    In particular, ${\cal L}_A^{\alpha}$ is a modular lattice with respect to
the partial order $(X,Y) \preceq (X',Y') \Leftrightarrow X \supseteq X', Y \subseteq Y'$,  
where the minimum  and maximum elements are given by $(X_{\alpha-1},Y_{\alpha-1})$ and $(X_\alpha,Y_\alpha)$, respectively.
\end{Lem}
For each $\alpha \in [\theta]$, consider a maximal chain (flag) of ${\cal L}_A^{\alpha}$:
\begin{eqnarray*}
&& X_{\alpha-1} = X_{\alpha,0} \supset X_{\alpha,1} \supset \cdots \supset X_{\alpha,\theta_\alpha} = X_\alpha, \\ 
&& Y_{\alpha-1} =Y_{\alpha,0} \subset Y_{\alpha,1} \subset \cdots \subset Y_{\alpha,\theta_\alpha} = Y_\alpha,
\end{eqnarray*}
where the length $\theta_\alpha$ of the chain is uniquely determined by the Jordan-Dedekind chain condition. 
The union $\bigcup_{\alpha=1}^\theta \bigcup_{\beta=0}^{\theta_{\alpha}} \{ (X_{\alpha,\beta},Y_{\alpha,\beta})\}$ is a maximal chain of
the whole lattice ${\cal L}_A := \bigcup_{\alpha=1}^\theta {\cal L}_A^{\alpha}$, and is called a {\em DM-flag}.
Its subset ${\cal E}_{A}$ is called the {\em coarse DM-flag}, which is uniquely determined by $A$. 
From a DM-flag, we obtain a simultaneous block upper triangular form of $A$ as follows.
Consider $g \in GL_n$ including, as row vectors, a basis of $X_{\alpha,\beta}$ for each $\alpha, \beta$.
Similarly, consider $h \in GL_m$ including, as row vectors, a basis of $Y_{\alpha,\beta}$ for each $\alpha,\beta$.
Suppose that they are positioned in the last rows for $g$ and first rows for $h$.
Then, the matrices $B_\ell  = g A_\ell h^{\dagger}$ are simultaneously block-triangularized,  
as in the right of Figure~\ref{fig:diagram}.
We call $B = (B_\ell)$ a {\em DM-decomposition}\footnote{The classical DM-decomposition restricts
${\cal S}_{A}$ to coordinate subspaces and 
${\cal L}_{A}$ to the sublattice of the coordinate subspaces $X,Y$ maximizing $\dim X+ \dim Y$, 
where $g,h$ are chosen as permutation matrices.
In this setting, 
a block-triangular form obtained by using the maximal chain of the entire ${\cal L}_A$
was considered by N. Tomizawa (unpublished) in the development of principal partitions in the 1970's; 
see \cite[Section 3]{HayashiHiraiSakabe}.
For this reason, our decomposition may be more precisely called a {\em DMT-decomposition}.
} of $A$.
When  $g$ (resp. $h$) is restricted to span only $X_\alpha$ (resp. $Y_\alpha$), 
it is called a {\em coarse DM-decomposition of $A$}.  

For abuse of notation, $X_\alpha$, $X_{\alpha,\beta}$, $Y_\alpha$, and $Y_{\alpha,\beta}$
also denote 
the index sets of the corresponding rows and columns of $B$.
Define ordered partitions $(I_\alpha)$ of $[n]$, $(J_\alpha)$ of $[m]$ 
and their refinements $(I_{\alpha,\beta})$, $(J_{\alpha,\beta})$ by 
\begin{eqnarray}
    && I_\alpha := X_{\alpha-1}\setminus X_{\alpha}, \quad J_\alpha := Y_{\alpha} \setminus Y_{\alpha-1}\quad (\alpha \in [\theta]), \\
    && I_{\alpha,\beta} := X_{\alpha,\beta-1} \setminus X_{\alpha,\beta},\quad J_{\alpha,\beta} := Y_{\alpha,\beta} \setminus Y_{\alpha,\beta-1}\quad (\beta \in [\theta_\alpha]).
\end{eqnarray}
Let $\hat B = (\hat B_{\ell})$ denote the matrix tuple
of block-diagonal matrices 
obtained from $B_{\ell}$ by replacing each (upper) off-diagonal block 
$B_k[I_{\alpha,\beta}, J_{\alpha',\beta'}]$ $((\alpha,\beta) \neq (\alpha',\beta'))$ with the zero matrix.
We call $\hat B$ a {\em diagonalized DM-decomposition of $A$}.
A diagonalized version of a coarse DM-decomposition is defined analogously.

Let $n_{\alpha} := |I_\alpha|$ and $m_{\alpha} := |J_\alpha|$.
By convexity of $\Conv \phi({\cal S}_A)$, it holds
\begin{equation}\label{eqn:n_1/m_2<...}
\frac{n_{1}}{m_1} < \frac{n_{2}}{m_2} < \cdots < \frac{n_{\theta}}{m_\theta}.
\end{equation}
Define $(p^*,q^*) \in \RR^n \times \RR^m$ by
\begin{equation}\label{eqn:p*q*}
    p^* := - \frac{1}{n}{\bf 1} + \frac{1}{C_A}\sum_{\alpha=1}^{\theta}\frac{m_\alpha}{n_\alpha + m_{\alpha}}{\bf 1}_{I_{\alpha}},\quad 
q^* := - \frac{1}{m}{\bf 1} +  \frac{1}{C_A}\sum_{\alpha=1}^{\theta} \frac{n_\alpha}{n_\alpha + m_{\alpha}}{\bf 1}_{J_{\alpha}}, 
\end{equation}
where the constant $C_A$ is defined by
\begin{equation}\label{eqn:C_A}
C_A := \sum_{\alpha=1}^{\theta}  \frac{n_{\alpha} m_\alpha }{n_\alpha + m_{\alpha}} 
\quad \leq \frac{nm}{n+m},
\end{equation}
where the inequality is seen from concavity of the harmonic mean $(x,y) \mapsto 2 (1/x+1/y)^{-1}$.
We see 
from (\ref{eqn:n_1/m_2<...})--(\ref{eqn:C_A}) 
that $(p^*,q^*)$ belongs to the positive Weyl chamber:
\begin{equation}\label{eqn:property_p*q*}
p_1^* \geq p_2^* \geq \cdots \geq p_n^*, \quad q_1^* \leq q_2^* \leq \cdots \leq q_m^*, \quad {\bf 1}^\top {p^*}  = {\bf 1}^\top {q^*} = 0.
\end{equation}
Recalling $P_n^1 := P_n \cap SL_n$, 
define $(G^*,H^*) \in \mathfrak{p}_n^1 \times  \mathfrak{p}_m^1 = T_{I,I} (P_n^1 \times P_m^1)$ by 
\begin{equation}
G^* := (\sigma^*)^{\dagger} \diag (- p^*) \sigma^* ,\quad H^* :=  (\tau^*)^{\dagger} \diag (-q^*) \tau^*, 
\end{equation}
where $\sigma^*$ is a unitary matrix having a basis of $X_{\alpha}$ in the last $n_{\alpha}$ rows and
$\tau^*$ is a unitary matrix having 
a basis of $Y_{\alpha}$ in the first $m_{\alpha}$ rows.
By using these notions, we give a solution to Problem~\ref{prob:ABC} (A), (B):
\begin{Thm}\label{thm:mnp_Q_A}
\begin{itemize}
\item[(1)] 
	$(p^*,q^*)$ is the minimum-norm point of $\varDelta_A$, and
\item[(2)] $(G^*,H^*)/\|(G^*,H^*)\|$ is the unique minimizer of $f^{\infty}_A$ over $S_{I,I}(P_n^1 \times P_m^1)$, 
where it holds 
\end{itemize}
\begin{equation}\label{eqn:equality}
\|(p^*,q^*)\|^2 = - f^{\infty}_A (G^*,H^*) = \frac{1}{C_A} - \frac{1}{n}- \frac{1}{m}. 
\end{equation}
\end{Thm}
\begin{Cor}\label{cor:mnp_Q_A}
Let $(g_k,h_k)$ and $(x_k,y_k)$ be the sequences in (\ref{eqn:DGR_F_A_G}) and (\ref{eqn:DGR_F_A}), respectively. 
\begin{itemize}
\item[(1)] $\spec \mu(g_k A h_k^{\dagger})$ converges to $(p^*,q^*)$ for $k \to \infty$.
\item[(2)] $(x_k,y_k)$ converges, in cone topology, to  $(G^*,H^*)/\|(G^*,H^*)\|$. 
More precisely, the sequence $(G_k,H_k)$ defined by $(x_k,y_k) = (e^{t G_k/L}, e^{t H_k/L})$
converges to $(G^*,H^*)$ for $k \to \infty$.
\end{itemize}
\end{Cor}

\begin{proof}[Proof of Theorem~\ref{thm:mnp_Q_A}]
We first show (\ref{eqn:equality}).
From the definitions of $(p^*,q^*)$ and $C_A$, we have
\[
\|(p^*,q^*)+({\bf 1}/n,{\bf 1}/m)\|^2 = \frac{1}{C^2_A}\sum_{\alpha=1}^\theta 
\frac{n_\alpha m_\alpha^2}{(n_\alpha+ m_\alpha)^2} +  \frac{m_\alpha n_\alpha^2}{(n_\alpha+ m_\alpha)^2} 
		= \frac{1}{C^2_A}\sum_{\alpha=1}^\theta 
		\frac{n_\alpha m_\alpha}{n_\alpha+ m_\alpha} = \frac{1}{C_A}.
\]
By the last equation in (\ref{eqn:property_p*q*}), we have
$$
\|(p^*,q^*)\|^2 =  \|(p^*,q^*)+ ({\bf1 }/n, {\bf 1}/m)\|^2 - \|({\bf1 }/n, {\bf 1}/m)\|^2 = 1/C_A - 1/n-1/m\ (> 0).
$$
On the other hand, 
$B = \sigma^* A(\tau^*)^{\dagger}$ is a coarse DM-decomposition, that is, 
$(\sigma^* A_{\ell} (\tau^*)^{\dagger})_{ij} = 0$ for each $(i,j) \in I_{\alpha} \times J_{\alpha'}$ with $\alpha > \alpha'$.
By (\ref{eqn:F_A^infty}) in the proof of Theorem~\ref{thm:scalability_limit}, 
the value of the recession function $f_{A}^{\infty}(G^*,H^*)$ is given by
\begin{equation}\label{eqn:f_A^infty(G*,H*)}
f_{A}^{\infty}(G^*,H^*) = \max \{- p^*_i- q_j^* \mid \exists \ell, (i,j) \in I_{\alpha} \times J_{\alpha'}:\alpha \leq \alpha',  (\sigma^* A_{\ell} (\tau^*)^{\dagger})_{ij} \neq 0 \}.
\end{equation}
Observe from (\ref{eqn:n_1/m_2<...})--(\ref{eqn:C_A}) that
\begin{equation}\label{eqn:=<>}
- p_i^* - q^*_j 
\left\{ 
\begin{array}{ll}
= 1/n + 1/m - 1/{C_A} & {\rm if}\  (i,j) \in I_{\alpha} \times J_{\alpha}, \\
< 1/n + 1/m - 1/{C_A}   &  {\rm if}\ (i,j)\in I_{\alpha} \times J_{\alpha'}: \alpha < \alpha',\\
> 1/n + 1/m - 1/{C_A}  &  {\rm if}\  (i,j) \in I_{\alpha} \times J_{\alpha'}: \alpha > \alpha'. 
\end{array} \right.
\end{equation}
Hence, the maximum in (\ref{eqn:f_A^infty(G*,H*)}) is attained by the index of any nonzero element of 
any diagonal block of $\sigma^* A_{\ell} (\tau^*)^{\dagger}$, which implies  
$f^{\infty}_A(G^*,H^*) = 1/n + 1/m - 1/C_A$, and (\ref{eqn:equality}).

       To complete the proof, it suffices to show $(p^*,q^*) \in \varDelta_A$ 
       since  $(p^*,q^*)$ and $(G^*,H^*)/\|(G^*,H^*)\|$ would attain
       $\inf_{(p,q)\in \varDelta_{A}} \|(p,q)\| = \sup_{(G,H) \in B_
       {I,I}} - f^{\infty}_A(G,H)$.
       This is done in the next proposition.
\end{proof}

\begin{Prop}\label{prop:hatB_in_Q_A}
Let $\hat B$ be a diagonalized DM-decomposition of $A$.
\begin{itemize}
\item[(1)] $\hat B$ is exactly $(p^*+{\bf 1}/n, q^*+{\bf 1}/m)$-scalable.  
\item[(2)] $[\hat B] \in \overline{[SL_n \cdot A \cdot SL_m]}$.
\end{itemize}
In particular, it holds $(p^*,q^*) \in \varDelta_A$.
\end{Prop}

\begin{proof}
(1). 
We first show:
\begin{Clm}
    $B[I_{\alpha,\beta},J_{\alpha,\beta}]$ 
	is exactly $({\bf 1}/|I_{\alpha,\beta}|,  {\bf 1}/|J_{\alpha,\beta}|)$-scalable.
\end{Clm}
\begin{proof}[Proof of Claim]
We can assume that $A$ is already 
equal to a DM-decomposition $B$, 
where all $X_{\alpha,\beta}$, $Y_{\alpha,\beta}$ are coordinate subspaces.
Suppose indirectly that $B[I_{\alpha,\beta},J_{\alpha,\beta}]$ is not exactly 
$({\bf 1}/|I_{\alpha,\beta}|,  {\bf 1}/|J_{\alpha,\beta}|)$-scalable.
Then, by Theorem~\ref{thm:Gurvits2}, 
there is nontrivial $(Z,W) \in {\cal S}_{B[I_{\alpha,\beta},J_{\alpha,\beta}]}$ 
such that $(1/|I_{\alpha,\beta}|)\dim Z+ (1/|J_{\alpha,\beta}|)\dim W \geq 1$. 
Then $(X_{\alpha,\beta} + Z, Y_{\alpha,\beta-1} + W)$ belongs to ${\cal S}_{A}$.
However, $\phi(X_{\alpha,\beta} + Z, Y_{\alpha,\beta-1} + W)$ goes beyond 
$\Conv \phi({\cal S}_A)$ or
lies on the interior of the segment between 
$\phi(X_{\alpha,\beta-1},Y_{\alpha,\beta-1})$ and $\phi(X_{\alpha,\beta},Y_{\alpha,\beta})$.  
The former case is obviously impossible.
The latter case is also impossible due to the maximality 
of the chain $\{(X_{\alpha,\beta},Y_{\alpha,\beta})\}$ in ${\cal L}_A$.
\end{proof}
We observe from  $n_{\alpha}/m_{\alpha} = |I_{\alpha,\beta}|/|J_{\alpha,\beta}|$ that $(m_{\alpha},  n_{\alpha})$ 
is a constant multiple of $(1/|I_{\alpha,\beta}|,  1/|J_{\alpha,\beta}|)$.
By the claim,  
for each $\alpha,\beta$, we can choose
scaling matrices $g_{\alpha,\beta},h_{\alpha,\beta}$  
to make $B[I_{\alpha,\beta},J_{\alpha,\beta}]$
an exact $(1/\{C_{A}(n_{\alpha}+ m_{\alpha})\})(m_{\alpha}{\bf 1},  n_{\alpha}{\bf 1})$-scaling. 
Then, for $g := \bigoplus_{\alpha,\beta} g_{\alpha,\beta}$, 
$h := \bigoplus_{\alpha,\beta} h_{\alpha,\beta}$,   
the scaling $g\hat{B}h^{\dagger}$ is a desired $(p^*+{\bf 1}/n,q^*+{\bf 1}/m)$-scaling.  

(2). Let $B$ 
be a DM-decomposition of $A$, where $B \in SL_n \cdot A \cdot SL_m$. 
For each $\alpha,\beta$ and $t > 0$, by $B[X_{\alpha,\beta}, Y_{\alpha,\beta}] = O$, it holds
\begin{equation}\label{eqn:Bije^-t}
(e^{t \diag {\bf 1}_{X_{\alpha, \beta}}} B e^{t \diag {\bf 1}_{Y_{\alpha, \beta}} - {\bf 1}})_{ij}
= \left\{ 
\begin{array}{ll}
B_{ij}e^{-t} & {\rm if}\ i \not \in X_{\alpha, \beta}, j \not \in Y_{\alpha, \beta}, \\
B_{ij}  & {\rm otherwise}.
\end{array}
\right.
\end{equation}
Let $R := \sum_{\alpha,\beta} |X_{\alpha,\beta}|/n$ and $S := \sum_{\alpha,\beta} (|Y_{\alpha,\beta}|-m)/m$. For $t> 0$, define $a_t \in SL_n$ and $b_t \in SL_m$ by 
	\begin{equation*}
	a_t := e^{- tR} e^{t \diag \sum_{\alpha,\beta} {\bf 1}_{X_{\alpha,\beta}}},\quad 
    b_t := e^{-tS }e^{t \diag \sum_{\alpha,\beta} {\bf 1}_{Y_{\alpha,\beta}}-{\bf 1}}.
	\end{equation*}
	By (\ref{eqn:Bije^-t}), the scaling $a_t B b_t$ is written as
	\[
a_t B b_t = e^{-(R+S)t} (\hat B + E_t)
	\]
    for the diagonalized DM-decomposition $\hat B$ of $B$ and matrix $E_t$ converging to zero for $t \to \infty$.
 This implies that $\lim_{t \to \infty}[a_t B b_t] = \lim_{t \to \infty} [\hat B+ E_t] = [\hat B] \in  \overline{[SL_n \cdot A \cdot  SL_m]}$.
 Since $\hat B$ admits an exact $(p^*+{\bf 1}/n, q^*+{\bf 1}/m)$-scaling $B^* = g\hat Bh^{\dagger}$.
 By Lemma~\ref{lem:(p,q)-scalable} and ${\bf 1}^{\top} p^* = {\bf 1}^{\top}q^* = 0$, 
 we conclude that $(p^{*},q^*) \in \varDelta_A$.
\end{proof}
Now the sequence of the scaled matrices along the gradient-descent trajectory accumulates to
the $SU_n \times SU_m$-orbit of a diagonalized DM-decomposition $\hat B$, providing a (partial) solution of Problem~\ref{prob:ABC}~(C): 
\begin{Thm}\label{thm:limit_in_P(V)}
Let $\hat B$ be a diagonalized DM-decomposition of $A$, and let 
$B^*$ be a $(p^*+{\bf 1}/n,q^*+{\bf 1}/m)$-scaling of $\hat B$.
Then $[g_kAh_k^{\dagger}]$ accumulates to points in $[SU_n \cdot B^* \cdot SU_m]$ for $k \to \infty$.
\end{Thm}
 
\begin{proof}
It holds
$
\mu(B^*)= (
\diag p^*, \diag q^*)
$.
Thus, $B^*$ attains the infimum of 
$\|\mu(B)\|$ over $[B] \in \overline{[SL_n \cdot A \cdot SL_m]}$, which is also the limit of $\| \mu(g_kAh_k^{\dagger})\|$.
By the second Ness uniqueness theorem 
(Theorem~\ref{thm:Ness_second}), we have the claim.
\end{proof}
For the gradient flow $(g(t),h(t))$ of the Kempf-Ness function $F_A$ on the group $SL_n \times SL_m$,
due to the convergence theorem (Theorem~\ref{thm:limit}), 
$[g(t)Ah(t)^{\dagger}]$ converges to a point $\sigma B^* \tau^{\dagger}$ for some $\sigma \in SU_n$, $s\tau \in SU_m$.

Although $B^*$ is also a diagonalized DM-decomposition of $A$, 
it is not clear how to remove the unitary indeterminacy from $[g_kA h_k^{\dagger}]$
and to extract the DM-structure of $B^*$.
This is possible for the coarse DM-structure as follows:
\begin{Thm}\label{thm:convergence_to_cDM}
Let $(G_k,H_k)$ be the sequence defined by $(x_k,y_k) = (e^{kG_k/L},e^{kH_k/L})$. 
Suppose that $G_k = \sigma_k^{\dagger} \diag a^k \sigma_k$ 
	and $H_k = \tau_k^{\dagger} \diag b^k  \tau_k$ 
    for unitary matrices $\sigma_k$, $\tau_k$ and nondecreasing and nonincreasing 
    vectors $a^k$ and $b^k$, respectively.
 Then $\sigma_k A \tau_k^{\dagger}$ accumulates to coarse DM-decompositions.
    The convergence is linear in the following sense: 
    There are $c > 0$, $M > 0$ such that for all $k \geq M$, $\ell \in [N]$
    it holds
    \[
    |(\sigma_k A_{\ell} \tau_k^\dagger)_{ij}| \leq e^{-ck} \quad 
    ((i,j) \in I_{\alpha} \times J_{\alpha'}: \alpha > \alpha').
    \] 
\end{Thm}


\begin{proof}

By Theorems~\ref{thm:discrete_main} and \ref{thm:mnp_Q_A} and
Lemma~\ref{lem:f(x_i+1)-f(x_i)} (1), it holds
\begin{eqnarray}
&& -\frac{1}{L}\left(\frac{1}{C_A} - \frac{1}{n}-\frac{1}{m}\right) = \lim_{k \to \infty} - \frac{\|\nabla f_A(x_k,y_k)\|^2}{L} \nonumber \\ 
&& = \lim_{k \to \infty} f_A(x_{k+1},y_{k+1})- f_A(x_k,y_k) = \lim_{k \to \infty} \frac{f_A(x_k,y_k)}{k}, \label{eqn:F_A/k}
\end{eqnarray}
where the final equality follows from (\ref{eqn:exercise1}) for $a_k := f_A(x_{k+1},y_{k+1})- f_A(x_k,y_k)$. 

Since
	$
	e^{f_A(x_k,y_k)} = \trace \sum_\ell x_k A_\ell y_k A_\ell^{\dagger} 
	= \sum_{\ell, i,j}  |(\sigma_k A_\ell \tau_k^\dagger)_{ij}|^2 e^{(a^k_i + b^k_j)k/L},
	$
we have 
 \[
 \sum_{\ell, i,j} |(\sigma_k A_\ell \tau_k^\dagger)_{ij}|^2 e^{(a^k_i + b^k_j)k/L - f_A(x_k,y_k)} = 1.
 \]
Suppose that the index $(i,j)$ is in a lower triangular block. 
By $(a^k,b^k) \underset{k \to \infty}{\rightarrow} - (p^*,q^*)$ (Corollary~\ref{cor:mnp_Q_A}~(2)) and (\ref{eqn:F_A/k}),  it holds 
\[
\frac{(a^k_i + b^k_j)k/L - f_A(x_k,y_k)}{k}\quad \underset{k \to \infty}{\longrightarrow} \quad \frac{1}{L}\left(- p^*_i - q^*_j - \frac{1}{n}-\frac{1}{m} + \frac{1}{C_A} \right) > 0,
\]
where the inequality follows from (\ref{eqn:=<>}).
Therefore, for some $c' > 0$ and $M' > 0$, 
it holds $(a^k_i  + b^k_j)k/L- f_A(x_k,y_k) \geq c'k$
for all $k > M'$.
Then $|(\sigma_k A_\ell \tau_k^\dagger)_{ij}|^2e^{c'k} \leq 1$ for all $k \geq M'$.
\end{proof}
\begin{Rem}\label{rem:mu(x_k^{1/2}Ay_k^{1/2})}
    Suppose that $\mu(x_k^{1/2}Ay_k^{1/2})$ converges, 
    or more strongly, the convergence of Question~\ref{que:convergence_D} is true. Then it holds
    $
    \lim_{k \to \infty} \|\mu(x_k^{1/2}Ay_k^{1/2}) + (G_k,H_k)\| = 0.
    $
    This implies 
    \begin{equation}\label{eqn:A^(k)}
    \lim_{k \to \infty} \|\mu(e^{\diag a^k/2} \sigma_k A \tau_k^{\dagger} e^{\diag b^k/2}) + (\diag a^k,\diag b^k)\| = 0.
    \end{equation}
    Since $(a^k,b^k) \to -(p^*,q^*)$, the scaling sequence
    $A^{(k)} := (e^{\diag a^k/2} \sigma_k A \tau_k^{\dagger} e^{\diag b^k/2})/\|g_k Ah_k\|$
    accumulates to $(p^* + {\bf 1}/n, q^* + {\bf 1}/m)$-scalings.
    From the coarse DM-structure of $\sigma_k A \tau_k^{\dagger}$ in the limit, one can see  
    that $A^{(k)}$ accumulates to diagonalized coarse DM-decompositions.
    Although our numerical experiment supports such convergence, 
    our results imply only $\liminf_{k \to \infty} = 0$ in~(\ref{eqn:A^(k)}).  
\end{Rem}
We end this subsection with some implications of these results.
%
\paragraph{On finding a destabilizing 1-PSG.}
Suppose that $A$ is not $({\bf 1}/n,{\bf 1}/m)$-scalable.
Consider $(X^*,Y^*) \in {\cal E}_A$ 
mapped to the extreme point $(r^*,s^*)$ of $ 
\Conv \phi({\cal S}_A)$ with the property that 
it maximizes $r$ among all extreme points $(r,s)$ maximizing $r+s$.
The subspace pair $(X^*,Y^*)$ violates (iii) in Theorem~\ref{thm:Gurvits1}
and is a special certificate of unscalability, called {\em dominant} in~\cite{FranksSomaGoemans_SODA2023}.
By Theorem~\ref{thm:convergence_to_cDM}, 
after a large number $k$ of iterations,  
the last $r^*$ rows of $\sigma_k$ and the first $s^*$ rows of 
$\tau_k$ become bases of 
an {\em $\epsilon$-approximate} dominant pair $(X^*_{\epsilon},Y^*_{\epsilon})$ 
in the sense that $|u^{\top}A_\ell \bar v| \leq \epsilon$
for all $\ell$ and all unit vectors $u \in X^*_{\epsilon}, v\in Y^*_{\epsilon}$.
Franks, Soma, and Goemans~\cite{FranksSomaGoemans_SODA2023} 
devised a procedure  
to round such an $e^{- p(n,m,N,b)}$-approximate dominant pair  
into the exact dominant pair $(X^*,Y^*)$, 
where $p$ is a polynomial and
$b$ is the bit complexity of $A$.  
Hence, if we would establish {\em global} linear convergence in Theorem~\ref{thm:convergence_to_cDM}, 
a polynomial number of iterations of 
gradient descent~(\ref{eqn:DGR_F_A}) 
would suffice to recover 
the dominant pair and a destabilizing 1-PSG.

\paragraph{Matrix scaling case.}
An $n \times m$ matrix $M = (a_{ij})$ is viewed  
as a matrix tuple $A = ( a_{ij} e_ie_{j}^{\top})_{ij: a_{ij} \neq 0}$.
Consider the left-right action on $A$, in which the group is restricted to 
the subgroup $ST_n \times ST_m \subseteq SL_n \times SL_m$ 
consisting of diagonal matrices. 
The corresponding scaling problem is nothing but 
the matrix scaling problem of the nonnegative matrix $(|a_{ij}|^2)$; see Section~\ref{subsec:Euclidean}.
The above results are also applicable to this setting. 
Indeed, the gradient $\nabla f_A$ is a pair of diagonal matrices.
Then, the gradient flow/descent belongs to the diagonal subspace in $P_n^1 \times P_m^1$, 
and is viewed as
the gradient flow/descent for the geometric programming objective~(\ref{eqn:matrixscaling}) in matrix scaling.
Here,  
all subspaces $X_{\alpha},Y_{\alpha}, X_{\alpha,\beta},Y_{\alpha,\beta}$ 
are coordinate subspaces. Hence a DM-decomposition $B$ is obtained by row and column permutations, and 
is equivalent to the original (extended) 
DM-decomposition of $M$.
In Remark~\ref{rem:mu(x_k^{1/2}Ay_k^{1/2})}, 
the unitary matrices $\sigma_k$ and $\tau_k$ are permutation matrices, 
and all lower triangular blocks of $A^{(k)}$ become zero matrices after finitely many iterations.
Also, all upper triangular blocks of $A^{(k)}$ converge to zero matrices.
In particular, the expected convergence to the diagonalized DM-decomposition $\hat B$ is true.
This convergence property is almost the same as the one for the Sinkhorn algorithm.
Indeed, \cite{HayashiHiraiSakabe} showed that 
ths limit ({\em Sinkhorn limit}) oscillates between 
the $({\bf 1},\sum_{\alpha} (n_\alpha/m_\alpha){\bf 1}_{J_{\alpha}})$-scaling $B^*_{\rm r}$ and $(\sum_{\alpha} (m_\alpha/n_\alpha){\bf 1}_{I_{\alpha}}, {\bf 1})$-scaling $B^*_{\rm c}$ of $\hat B$.

\paragraph{On the limit of the operator Sinkhorn algorithm.}
This suggests an expectation of the limiting behavior of 
the {\em operator Sinkhorn algorithm (Gurvits' algorithm),}
the standard algorithm for the operator scaling problem.
The operator Sinkhorn algorithm 
is viewed as alternating minimization of $f_A(x,y)$, where
each step  scales $A \to gA$ with $\mu(A) = (O, \ast)$ 
and $A \to Ah^{\dagger}$ with $\mu(A) = (\ast, O)$ alternatively.
When it is applied to the $(p^*+{\bf 1}/n, q^*+{\bf 1}/m)$-scaling $B^*$ of a diagonalized DM-decomposition $\hat B$, 
the resulting scaling sequence oscillates between  
the $({\bf 1},\sum_{\alpha} (n_\alpha/m_\alpha){\bf 1}_{J_{\alpha}})$-scaling and 
$(\sum_{\alpha} (m_\alpha/n_\alpha){\bf 1}_{I_{\alpha}}, {\bf 1})$-scaling of $B^*$.
With the view of Theorem~\ref{thm:limit_in_P(V)} and the matrix scaling case above, 
it is reasonable to conjecture that it oscillates between orbits $U_n\cdot B_{\rm r}^* \cdot U_m$
and $U_n \cdot B_{\rm c}^*\cdot U_m$, where $B_r^*$ (resp. $B_c^*$) is a 
$({\bf 1},\sum_{\alpha} (n_\alpha/m_\alpha){\bf 1}_{J_{\alpha}})$-scaling (resp. $(\sum_{\alpha} (m_\alpha/n_\alpha){\bf 1}_{I_{\alpha}}, {\bf 1})$-scaling) of $\hat B$.

\subsection{Kronecker form of a matrix pencil}\label{subsec:pencil}
Finally, we discuss the special case of $N = 2$, i.e., 
$A = (A_1,A_2)$.
In this case, $A$ is naturally identified with
a {\em matrix pencil} $s A_1 + A_2 \in \CC(s)^{n \times m}$, where $s$ is an indeterminate.
Here we reveal a connection to the {\em Kronecker canonical form} of $s A_1 + A_2$, and suggest 
a new numerical method for finding
the Kronecker structure based on gradient descent.

A pencil $s A_1+ A_2$ is called {\em regular} 
if $n=m$ and 
$\det (s A_1+A_2) \neq 0$ for some $s \in \CC$.
Otherwise, 
$s A_1+A_2$ is called {\em  singular}. 
For simplicity, we assume (again) that 
$\ker A_1 \cap \ker A_2 = \{0\}$ and $\ker A^{\dagger}_1 \cap \ker A^{\dagger}_2 = \{0\}$. 
The Kronecker form is a canonical form of a (singular) pencil
under transformation $(sA_1+A_2) \to g(sA_1+A_2)h^{\dagger}$ by $g \in GL_n$, $h \in GL_m$. The standard reference of the Kronecker form is 
\cite[Chapter XII]{Gantmacher}; see also \cite[Section 5.1.3]{MurotaMatrixMatroid} for its importance in systems analysis.  
For a positive integer $\epsilon$, 
define $\epsilon \times (\epsilon +1)$ 
matrix $L_\epsilon$ by
\begin{equation*}
    (L_\epsilon)_{ij} := 
    \left\{
    \begin{array}{ll}
    1 & {\rm if}\ j=i, \\
    s & {\rm if}\ j=i+1, \\
    0 & {\rm otherwise}.
    \end{array}
    \right.
\end{equation*}

\begin{Thm}[{Kronecker form; \cite[Chapter XII]{Gantmacher}}]\label{thm:KCF}
There are $g \in GL_n, h \in GL_m$ such that
\begin{equation}\label{eqn:KCF}
g(sA_1+A_2)h^{\dagger} =  L_{\epsilon_1} \oplus L_{\epsilon_2} 
\oplus \cdots \oplus L_{\epsilon_c} \oplus (sC+D)  \oplus L^{\dagger}_{\eta_d} \oplus L^{\dagger}_{\eta_{d-1}} 
\oplus \cdots \oplus L^{\dagger}_{\eta_1},
\end{equation}
where $sC+D$ is a regular pencil, 
and $\epsilon_1, \epsilon_2, \ldots, \epsilon_c$, $\eta_1,\eta_2,\ldots,\eta_d$ are positive integers determined as follows:
\begin{itemize}
\item $\epsilon_j$ is the minimum degree of a polynomial vector $x_j(s)$ in $\ker sA_1+A_2$ 
that is linearly independent from $x_{1}(s),x_{2}(s),\ldots, x_{j-1}(s)$ over $\CC(s)$.
\item $\eta_j$ is the minimum degree of a polynomial vector  
$y_j(s)$ in $\ker (sA_1+A_2)^{\dagger}$ 
that is linearly independent from $y_{1}(s),y_{2}(s),\ldots, y_{j-1}(s)$ over $\CC(s)$.
\end{itemize}
\end{Thm}
The indices $\epsilon_1 \leq \cdots \leq \epsilon_c, \eta_1 \leq \cdots \leq \eta_d$, called the {\em minimal indices}, are uniquely determined. 
If $n=m$ and $sA_1+A_2$ is singular, 
then the Kronecker form has a zero block 
with the sum of row and column numbers greater than $n$. 
Therefore, by Theorem~\ref{thm:Gurvits1}, we have:
\begin{Cor}
A pencil $sA_1 + A_2$ is regular if and only if $n = m$ and 
$(A_1,A_2)$ is $({\bf 1}/n, {\bf 1}/n)$-scalable.
\end{Cor}
We point out a further connection that the Kronecker form~(\ref{eqn:KCF})  
is viewed as almost a DM-decomposition.
Let $b$ denote the number of diagonal blocks of $gAh^{\dagger}$ in (\ref{eqn:KCF}).
For $\gamma \in [b]$, 
let $I_\gamma$ and $J_\gamma$ denote
the row and column index sets, respectively, of
the $\gamma$-th diagonal block of $gAh^{\dagger}$. 
Define $X_\gamma$ 
by the vector subspace 
spanned by the rows of $g$ of indices in 
$I_{\gamma+1} \cup I_{\gamma+2} \cup \cdots \cup I_{b}$.
Similarly, define $Y_\gamma$ by
the vector subspace 
spanned by the rows of $h$ having 
indices in 
$J_{1} \cup \cdots \cup J_{\gamma}$. 
We let $(X_0,Y_0) := (\CC^n, \{0\})$ (and $(X_b,Y_b) = (\{0\}, \CC^m))$.
Suppose that $s C+D (= g(sA_1+A_2)h^{\dagger}[I_{c+1},J_{c+1}])$ exists and is an $n_0 \times n_0$ upper triangular matrix.
Let $Z_\beta$ denote the vector space 
spanned by the rows of $g$ having the last $n_0 - \beta$ indices in $I_{c+1}$, 
and let $W_\beta$ denote the vector space 
spanned by the rows of $h$ having the first $\beta$ indices in $J_{c+1}$.
Let $X_{c,\beta} := X_{c+1}+Z_\beta$ and $Y_{c,\beta} := Y_{c}+W_\beta$, where $+$ is the direct sum.
Consider all indices $\gamma$ with $(|I_\gamma|,|J_\gamma|) \neq (|I_{\gamma+1}|,|J_{\gamma+1}|)$, and suppose that they 
are ordered as $0 =: \gamma_0 < \gamma_1 < \cdots < \gamma_\theta := b$.

\begin{Prop}
\begin{itemize}
\item[(1)] $\{(X_{\gamma_\alpha},Y_{\gamma_\alpha})\}^\theta_{\alpha=0}$ is the coarse DM-flag of $(A_1,A_2)$.
\item[(2)] Suppose that $s C+D$ is an $n_0 \times n_0$ upper triangular pencil. Then the union of 
$\{(X_{\gamma}, Y_{\gamma})\}_{\gamma= 0}^b$ and 
$\{ (X_{c,\beta},Y_{c,\beta})\}_{\beta = 1}^{n_0-1}$ is a DM-flag of $(A_1,A_2)$.
\end{itemize}
\end{Prop}
\begin{proof}
(1). Suppose that ${\cal E}_A$ consists of $(X'_{\alpha},Y'_{\alpha})$ 
for $\alpha = 0,1,2\ldots,\theta'$, arranged as in (\ref{eqn:(X_alpha,Y_alpha)}).
We show $(X'_{\alpha},Y'_{\alpha}) = (X_{\gamma_\alpha},Y_{\gamma_\alpha})$ 
for $\alpha = 0,1,2\ldots,\theta' = \theta$.
Consider the convex hull $K_A$ of $(0,0)$ and $\phi(X_\gamma,Y_\gamma)$ for all $\gamma$.
Then $K_A$ belongs to $\Conv \phi({\cal S}_A)$, and the maximal faces of $K_A$ 
are composed of the line segments connecting points
$\phi(X_\gamma,Y_\gamma)$ from $\gamma=0$ to $b$
with bending points $\phi(X_{\gamma_\alpha},Y_{\gamma_\alpha})$.

We show  $K_A = \Conv \phi({\cal S}_A)$ by induction on the number $b$ of diagonal blocks.
Consider the base case $b=1$
where the Kronecker form consists of a single block.
It suffices to show ${\cal E}_A = \{(\CC^{n},0), (0, \CC^{m})\}$. 
Suppose that $sA_1+A_2$ is an $n_0 \times n_0$ regular pencil $sC+D$.
By regularity, there is no $(X,Y) \in {\cal S}_A$ with $\dim X+\dim Y > n_0$ 
(otherwise $sA_1+A_2$ is singular over $\CC(s)$).
This means no point in $\phi({\cal S}_A)$ beyond the line segment between $(n_0,0)$ and $(0,n_0)$.
Therefore, we have ${\cal E}_A = \{(\CC^{n},0), (0, \CC^{m})\}$.
Suppose that $sA_1+A_2 = L_n$.
Suppose to the contrary that there is $(X,Y) \in {\cal S}_A$  with $\dim X/n+\dim Y/(n+1) > 1$.  
By basis change,  we may assume that 
\[
sA_1 + A_2 = \left( 
\begin{array}{cc}
B & C \\
O & D
\end{array}
\right),
\]
where $O$ is the $r \times s$ zero matrix for $(r,s):= (\dim X,\dim Y)$.
By $r \geq 1$ and $r + s \geq n+1$, 
$B$ is a pencil of $n-r$ rows and $s$ columns with $s > n-r$.
Then $\ker B$ 
contains a polynomial vector 
with degree at most $n-r < n$; use Cramer's formula to see this.
Necessarily, $\ker sA_1+A_2$ also has such a polynomial vector. 
This is a contradiction to Theorem~\ref{thm:KCF} ($\epsilon_1 = \epsilon_c = n$).  
The case $sA_1+A_2 = L_n^{\dagger}$ is similar.

Consider a general case of $b \geq 2$.
We can choose $\gamma^*, \alpha^*$ 
such that $0 < \gamma^* < b$, $0 < \alpha^* < \theta'$, 
and the line segment between 
$\phi(X_{\gamma^*},Y_{\gamma^*})$ and $\phi(X'_{\alpha^*},Y'_{\alpha^*})$
meets with $K_A$ only at $\phi(X_{\gamma^*},Y_{\gamma^*})$. 
%
Consider $(U,V)  := (X_{\gamma^*} + X'_{\alpha^*}, Y_{\gamma^*} \cap Y'_{\alpha^*})$ and 
 $(U',V')  := (X_{\gamma^*} \cap X'_{\alpha^*}, Y_{\gamma^*} + Y'_{\alpha^*})$.
 By the construction and 
 (\ref{eqn:modular_equality}), one of $\phi(U,V)$ and $\phi(U',V')$ is outside of $K_A$.
Suppose that $\phi(U,V) \not \in K_A$.
Consider the submatrix $A' := (sA_1+A_2)[\bigcup_{\gamma=1}^{\gamma^*} I_\gamma, \bigcup_{\gamma=1}^{\gamma^*} J_\gamma]$,   
that is also a Kronecker form with a smaller number of blocks. 
From $U \supseteq X_{\gamma^*}$, $V \subseteq Y_{\gamma^*}$,  and $\phi(U,V) \not \in K_A$,
it necessarily holds $K_{A'} \neq \Conv \phi({\cal S}_{A'})$. 
However, this is a contradiction to the inductive assumption.
The case $\phi(U',V') \not \in K_A$ is similar; 
consider the sub-Kronecker form $(sA_1+A_2)[\bigcup_{\gamma=\gamma^*+1}^b I_\gamma, \bigcup_{\gamma=\gamma^*+1}^b J_\gamma]$. 

(2). Observe that all integer points in the maximal faces of $\Conv \phi({\cal S}_A)$
are obtained by the images of $(X_\gamma,Y_\gamma)$ and $(X_{c,\beta},Y_{c,\beta})$.
This implies that $\{(X_{\gamma}, Y_{\gamma})\}_\gamma \cup \{ (X_{c,\beta},Y_{c,\beta})\}_\beta$ 
is a maximal chain of ${\cal L}_A$.
\end{proof}

The matrix pencil $g(sA_1+A_2)h^{\dagger}$ corresponding to 
a coarse DM-decomposition $g(A_1,A_2)h^{\dagger}$, which we call 
a {\em coarse Kronecker triangular form}, 
is a refinement of a {\em quasi-Kronecker triangular form} 
in~\cite{ThomasStephan2012} 
and {\em generalized Schur form} in \cite{DemmelKagstrom1993,VanDooren1979} 
if $g,h$ are unitary matrices and $s C+D$ is triangular.

Then, the convergence (Theorem~\ref{thm:convergence_to_cDM}) of
gradient descent (\ref{eqn:DGR_F_A}) can be applied as follows:
\begin{Thm}[Convergence to a coarse Kronecker triangular form]\label{thm:convergence_to_cKronecker}
Let $(x_k,y_k)$ be a solution of (\ref{eqn:DGR_F_A}).
Decompose $x_k = \sigma_k^{\dagger} e^{\diag a^k} \sigma_k$ 
	and $y_k = \tau_k^{\dagger} e^{\diag b^k} \tau_k$, 
 where $\sigma_k$ and $\tau_k$ are unitary matrices, 
 and $a^k$ and $b^k$ are nondecreasing and nonincreasing vectors, respectively.
    Then, 
   $\sigma_k (sA_1+A_2) \tau_k^{\dagger}$ accumulates to coarse Kronecker triangular forms,  
   where the convergence is linear in the same sense as in Theorem~\ref{thm:convergence_to_cDM}.
\end{Thm}
A coarse Kronecker triangular form is enough for determining 
the structure of the Kronecker form. 
Indeed, each (non-square) rectangular diagonal block is 
a $k\nu \times k(\nu+1)$ or 
$k(\nu+1) \times k\nu$ matrix
for some integers $k, \nu$, from which all minimal indices 
$\epsilon_1,\epsilon_2,\ldots,\epsilon_c$,
$\eta_1,\eta_2,\ldots,\eta_d$ can be identified.

The above theorem suggests an iterative method 
for determining the minimal indices of a singular pencil, 
which is based on simple gradient descent and is conceptually different from 
the existing algorithms, e.g., \cite{DemmelKagstrom1993,VanDooren1979}.
It is an interesting future direction to
develop a numerically stable algorithm based on this approach.

\section*{Acknowledgments}
We thank Shin-ichi Ohta, Harold Nieuwboer, and Michael Walter for discussion, Yuni Iwamasa, Taihei Oki and Tasuku Soma for comments, 
and Shun Sato for suggesting~\cite{SanzSernaZygalakis2020}.
We also thank the referees for numerous helpful comments.
The first author was supported by JSPS KAKENHI Grant Number JP21K19759, JP24K21315.
The second author was supported by the European Union (ERC Starting Grant SYMOPTIC, 101040907) and by the Deutsche Forschungsgemeinschaft (DFG, German Research Foundation, 556164098).

\bibliographystyle{plain}
\bibliography{gradientflow}

\end{document}